\documentclass[12pt]{article}
\usepackage[left=3cm,right=3cm, top=2.5cm,bottom=2.5cm,bindingoffset=0cm]{geometry}
\usepackage{cmap}

\usepackage{amsmath,amsthm,amssymb}

\usepackage{amscd}

\usepackage{hyperref}

\usepackage{MnSymbol} 

\usepackage[table]{xcolor}

\usepackage{wrapfig}
\usepackage{graphicx}

\makeatletter
\def\th@plain{%
  \thm@notefont{}
  \itshape 
}
\def\th@definition{%
  \thm@notefont{}
  \normalfont 
}
\makeatother

\newtheorem{lemma}{Lemma}[section]
\newtheorem{proposition}[lemma]{Proposition}

\newtheorem{remark-definition}[lemma]{Remark-Definition}
\newtheorem{theorem}[lemma]{Theorem}
\newtheorem{corollary}[lemma]{Corollary}

\newtheorem{proposition-conjecture}[lemma]{Proposition-conjecture}
\newtheorem{problem}[lemma]{Problem}
\newtheorem{assertion}[lemma]{Assertion}

\theoremstyle{definition}

\newtheorem{definition}[lemma]{Definition}
\newtheorem{remark}[lemma]{Remark}

\newcommand{\Ker}{\mathrm{Ker}\,}

\newcommand{\St}{\mathrm{St}}

\newcommand{\rk}{\mathrm{rk}\,}

\newcommand{\gl}{\mathrm{gl}\,}

\newcommand{\goth}{\mathfrak}

\newcommand{\codim}{\mathrm{codim}\,}

\newcommand{\Sing}{\mathsf{Sing}}

\newcommand{\svert}{\mathrm{v}_{tot}}

\newcommand{\nv}{n_{{v}}}

\newcommand{\nh}{n_{{h}}}

\newcommand{\dimO}{\dim \mathcal O_{\mathrm{reg}}}

\newcommand{\codimO}{\codim \mathcal O_{\mathrm{reg}}}

\newcommand{\dimSt}{\dim \mathrm{St}_{\mathrm{reg}}}

\ifdefined\C
 \renewcommand{\C}{\mathbb{C}}
\else
 \newcommand{\C}{\mathbb{C}}
\fi

\newcommand\CP{\mathbb{C}\mathbb{P}}

\ifdefined\N
 \renewcommand{\N}{\mathbb{N}}
\else
 \newcommand{\N}{\mathbb{N}}
\fi

\newcommand{\charp}{\chi}

\newcommand\g{\goth{g}}

\usepackage[numbers, sort]{natbib}

\usepackage{mathtools}
\mathtoolsset{showonlyrefs}

\title{Jordan--Kronecker invariants of Lie algebra representations: examples and computations}
\author{Ivan Kozlov\thanks{No Affiliation. E-mail: {\tt ikozlov90@gmail.com} }
}

\begin{document}

\date{}

\maketitle

\abstract{In these paper we compute Jordan-Kronecker invariants of Lie algebra representations, introduced earlier by A.\,V.~Bolsinov, A.\,M.~Izosimov and I.\,K.~Kozlov, for a number of representations. In particular, we compute them for 

\begin{itemize}

\item the sums of standard representations of $\operatorname{gl}(n)$, $\operatorname{sl}(n)$, $\operatorname{so}(n)$, $\operatorname{sp}(n)$, and the Lie algebra of upper triangular matrices $\operatorname{b}(n)$;

\item the standard representation of Lie algebra of strictly upper triangular matrices $\operatorname{n}(n)$;

\item and for the differential of the congruence  action of $\operatorname{GL}(n)$ and $\operatorname{SL}(n)$ on symmetric forms and skew-symmetric forms.
\end{itemize}

 }

\tableofcontents

\section{Introduction}
\label{intro}

The Jordan--Kronecker invariants for Lie algebra representation were introduced in \cite{BolsIzosKozl19}. They are defined as follows. Let  $\rho : \mathfrak{g} \to \gl (V)$ be a linear representation of a complex finite-dimensional Lie algebra $\mathfrak{g}$ on a finite-dimensional vector space $V$. To this representation and an arbitrary element $x \in  V$, one can naturally associate an operator \begin{align*} R_x :\mathfrak{g} \to V, \\ R_x (\xi) = \rho(\xi) x.\end{align*} Consider a pair of such operators $R_a, R_b$ and the pencil $R_a+\lambda R_b=R_{a+\lambda b}$ generated by them. It is well known that such a pencil can be completely characterized by a collection of simple numerical invariants (see Section~\ref{S:JK_Operator} for details). The Jordan--Kronecker invariants of $\rho$ are such invariants for a  pencil $R_{a+\lambda b}$ generated by a generic pair $(a,b)\in V\times V$. 

In this paper we calculate the Jordan--Kronecker invariants for the following representations:

\begin{enumerate}

\item Sum of m standard representations for

\begin{enumerate}

\item Lie algebra $\operatorname{gl}(n)$ (Theorem~\ref{T:JKSumStandardGLn})

\item Lie algebra $\operatorname{sl}(n)$ (Theorem~\ref{T:JKSumStandardSLn})

\item Lie algebras $\operatorname{so}(n)$ and $\operatorname{sp}(n)$ (Theorem~\ref{T:JKSumStandardSOn})

\item  Lie algebra of upper triangular matrices $\operatorname{b}(n)$ (Theorem~\ref{T:JKSumStandardBn})

\end{enumerate}

\item Standard representation of strictly upper triangular matrices $\operatorname{n}(n)$. (Theorem~\ref{T:JKSumStandardNn}).

\item The differential of congruence action $P, Q \to PQP^T$ on symmetric forms $S^2(V)$ and skew-symmetric forms $\Lambda^2(V)$. The Lie algebra is either $\mathfrak{g} = \gl(V)$ or $\operatorname{sl}(V)$. Here is the full list of theorems:

\begin{enumerate}

\item Lie algebra $\gl(V)$, symmetric forms $S^2(V)$ (Theorem~\ref{Th:GL_Sym_CongAct} ),

\item Lie algebra $\operatorname{sl}(V)$, symmetric forms $S^2(V)$ (Theorem~\ref{T:SL_SymCong}),

\item Lie algebra $\gl(V)$, skew-symmetric forms $\Lambda^2(V)$, $\dim V = 2n$ (Theorem~
\ref{T:GL_SkewEvenCong}),

\item Lie algebra $\operatorname{sl}(V)$, skew-symmetric forms $\Lambda^2(V)$, $\dim V = 2n$ (Theorem~\ref{T:SL_SkewEvenCong}),

\item Lie algebras $\gl(V)$ and $\operatorname{sl}(V)$, skew-symmetric forms $\Lambda^2(V)$, $\dim V = 2n+1$ (Theorem~\ref{T:GL_SkewOddCong} ).

\end{enumerate}

\end{enumerate}

Theorems~\ref{T:JKSumStandardSLn} and \ref{T:JKSumStandardSOn} (about standard representations of $\operatorname{sl}(n)$, $\operatorname{so}(n)$ and $\operatorname{sp}(n)$) were stated in \cite{BolsIzosKozl19}. In this paper we provide all the calculations.

All Lie algebras are finite-dimensional  and complex (although all results remain true for an arbitrary algebraically closed field $\mathbb{K}$ with $\operatorname{char} \mathbb{K} = 0$). For short, we use JK as an abbreviation of Jordan-Kronecker.

\begin{remark}
The concept of JK invariants of representations is closely related to the JK invariants of Lie algebras, introduced in \cite{BolsZhang}. Apart from this pioneering work and the references therein, one can see  \cite{Vor1, Vor2, Vor3, Vor4, Gar1, Gar2}. \end{remark}

\par\medskip
 
\textbf{Acknowledgements} The author would like to thank A.\,V.~Bolsinov for the invention of Jordan--Kronecker invariants and useful comments, A.\,M~Izosimov for his invaluable contribution to the creation of \cite{BolsIzosKozl19}.

\section{Jordan-Kronecker invariants of representations}
\label{S:JK_Operator}

In this section we recall the definition of the JK invariants and their basic properties. This is mostly a retelling of \cite{BolsIzosKozl19}. We leave it for the sake of completeness.

\subsection{Canonical form of a pair of linear maps}
\label{SubS:JK_Operator}

In this section, we recall the normal form theorem for a pair of linear maps.  We first state the theorem in the matrix form, and then discuss the corresponding invariants.

\begin{theorem}[on the Jordan-Kronecker normal form \cite{Gantmaher88}]
\label{T:JK_operator}
Consider two complex vector spaces
$U$ and $V$. Then for every two linear maps $A, B: U\to V$ there are bases in  $U$ and $V$ in which the matrices of the pencil $\mathcal P=\{A + \lambda B\}$ have the following block-diagonal form:
\begin{equation}
\label{Eq:JK_Operator}
{\footnotesize A + \lambda B =
\begin{pmatrix}
0_{m, n} &     &        &      \\
    & \!\!\! A_1 + \lambda B_1 &        &      \\
    &     & \!\!\!\! \ddots &      \\
    &     &        & \!\!\!  A_k + \lambda B_k  \\
\end{pmatrix},
}
\end{equation}
where $0_{m, n}$ is the zero $m \times n$-matrix, and each pair of the corresponding blocks   $A_i$ and $B_i$ takes one of the following forms:

1. Jordan block with eigenvalue $\lambda_0 \in \mathbb{C}$
 \[
 A_i =\left( \begin{matrix}
   \lambda_0 &1&        & \\
      & \lambda_0 & \ddots &     \\
      &        & \ddots & 1  \\
      &        &        & \lambda_0   \\
    \end{matrix} \right),
\quad  B_i=  \left( \begin{matrix}
    1 & &        & \\
      & 1 &  &     \\
      &        & \ddots &   \\
      &        &        & 1   \\
    \end{matrix} \right).
\]

2. Jordan block with eigenvalue $\infty$
\[
A_i = \left( \begin{matrix}
   1 & &        & \\
      &1 &  &     \\
      &        & \ddots &   \\
      &        &        & 1   \\
    \end{matrix}  \right),
\quad B_i = \left( \begin{matrix}
    0 & 1&        & \\
      & 0 & \ddots &     \\
      &        & \ddots & 1  \\
      &        &        & 0   \\
    \end{matrix}  \right).
 \]

  3. Horizontal Kronecker block 
\[
A_i= \left(
 \begin{matrix}
    0 & 1      &        &     \\
      & \ddots & \ddots &     \\
      &        &   0    & 1  \\
    \end{matrix}  \right),\quad
B_i = \left(\begin{matrix}
   1 & 0      &        &     \\
      & \ddots & \ddots &     \\
      &        & 1    &  0  \\
    \end{matrix}  \right)
\]

   4. Vertical Kronecker block
\[
A_i= \left(
  \begin{matrix}
  0  &        &    \\
  1   & \ddots &    \\
      & \ddots & 0 \\
      &        & 1  \\
  \end{matrix}
 \right), \quad
B_i = \left(
  \begin{matrix}
  1  &        &    \\
  0   & \ddots &    \\
      & \ddots & 1 \\
      &        & 0  \\
  \end{matrix}
 \right).
\]

 The number and types of blocks in decomposition  \eqref{Eq:JK_Operator} are uniquely defined up to permutation.
 \end{theorem}

\begin{remark} Unlike \cite{BolsIzosKozl19} (and similar to \cite{Gantmaher88}) there are no negative signs in the matrices $B_i$.  It doesn't affect much: there is a change of signs in some formulas (e.g. one should consider $\operatorname{Ker} (A- \lambda B)$ instead of $\operatorname{Ker} (A + \lambda B)$) but the JK invariants do not change. 
\end{remark}

It is convenient to regard the zero block $0_{m, n}$ in \eqref{Eq:JK_Operator} as a block-diagonal matrix that is composed of $m$ vertical Kronecker blocks of size  $1\times 0$ and $n$ horizontal Kronecker blocks of size $0\times 1$.

\begin{definition}
The \textit{horizontal indices} $\mathrm{h}_1, \dots, \mathrm{h}_p$ of the pencil $\mathcal P=\{A+\lambda B\}$ are defined to be the horizontal dimensions (widths) of horizontal Kronecker blocks: \[ A_i + \lambda B_i  = \underbrace{\left(
 \begin{matrix}
    \lambda & 1      &        &     \\
      & \ddots & \ddots &     \\
      &        &  \lambda    & 1  \\
    \end{matrix}  \right)}_{\mathrm{h}_i}.\] Similarly, the \textit{vertical indices} $\mathrm{v}_1, \dots, \mathrm{v}_q$  are the vertical dimensions (heights) of vertical blocks: \[ A_i + \lambda B_i  = \left. \left(
  \begin{matrix}
  \lambda  &        &    \\
  1   & \ddots &    \\
      & \ddots & \lambda \\
      &        & 1  \\
  \end{matrix}
 \right)\right\rbrace{\scriptstyle \mathrm{v}_i}.\] 
\end{definition}
In particular, in view of the above interpretation of the $0_{m,n}$ block, the first $m$ horizontal indices and first $n$ vertical indices are equal to $1$. 
We will denote the total number $p$ of horizontal indices by $\nh$, and the total number $q$ of vertical indices by $\nv$.
\begin{remark}\label{minIndices}
There also exist closely related notions of \textit{minimal row and column indices} (see \cite{Gantmaher88}). Namely, minimal column indices are equal to the numbers $h_i -1$, while minimal row indices are equal to $v_i - 1$. However, we find the terms \textit{horizontal and vertical indices} more intuitive and more suitable for our purposes.
\end{remark}

\subsection{Jordan--Kronecker invariants of Lie algebra representations}

Consider a finite-dimensional linear representation  $\rho: \goth g \to \gl(V)$ of a finite-dimensional Lie algebra $\goth g$.  To each point  $x\in V$, the representation $\rho$ assigns a linear operator  \begin{align*} R_x: &\goth g\to V, \\ R_x(\xi) = & \rho(\xi) x \in V. \end{align*}

\begin{definition} The \textbf{algebraic type} of a pencil $R_a + \lambda R_b$, $a, b \in V$ is the following collection of discrete invariants:

\begin{itemize}

\item the number of distinct eigenvalues of Jordan blocks,

\item the number and sizes of Jordan blocks associated with each eigenvalue,

\item horizontal and vertical Kronecker indices.

\end{itemize}
\end{definition}

Note that for Jordan blocks we can consider tuples of sizes for each eigenvalue. 

\begin{remark}
\label{Rem:GenLinCombGen}
The algebraic type of a pencil $R_a + \lambda R_b$ does not change under replacing $a$ and  $b$ with any linearly independent combinations of them $a'=\alpha a+\beta b$  and $b'=\gamma a + \delta b$.
\end{remark}

\begin{definition}
\textbf{The Jordan--Kronecker invariant} of $\rho$ is the algebraic type of a pencil $R_{a + \lambda b}$ for a generic pair $a, b \in V$.
\end{definition}

In particular, horizontal and vertical Kronecker indices of a generic pencil will be denoted by
 $\mathrm{h}_1(\rho),\dots,\mathrm{h}_p(\rho)$ and $\mathrm{v}_1(\rho), \dots, \mathrm{v}_q(\rho)$ and will be called {\it horizontal and vertical indices} of the representation $\rho$.

\subsubsection{Invariance under group action} 

The next statement that we will need below deals with the invariace of JK invariants under the corresponding group action.

\begin{proposition} \label{Prop:GroupPresJK} Assume that a representation of Lie algebra $\mathfrak{g}$ is the differential $d \phi$ of the Lie group representation $\phi: G \to \operatorname{GL}(V)$. Then for any $g \in G$ the JK invariants for pencils $R_{x+\lambda a}$ and $R_{\phi(g) (x +\lambda a)}$ coincide.  \end{proposition}

\begin{proof} Note that \[R_{\phi(g) x}(\xi)  = \phi(g)  R_x \left(\operatorname{Ad}_{g^{-1}} \xi\right). \] Indeed, any element of Lie algebra $\xi \in \mathfrak{g}$ has the from $\xi = \frac{d}{dt} h_t$, where $h_t \in G$. Thus \[ R_{\phi(g) x} (\xi) = \frac{d}{dt}  \phi\left( h_t \right) \phi(g) x = \phi(g) \frac{d}{dt} \phi\left(g^{-1} h_t g\right)  x = \phi(g)  R_x \left(\operatorname{Ad}_{g^{-1}} \xi\right).\] So, the following diagram commutes \[
\begin{CD}
\mathfrak{g}  @>R_{\phi(g) (x  + \lambda a)} >> V \\
@V \operatorname{Ad}_{g^{-1}} VV @AA\phi(g) A\\
\mathfrak{g} @>R_{x + \lambda a} >> V
\end{CD}
\] 
The JK decomposition for any pair of linear maps doesn't depend on a choice of bases in the spaces (i.e. they don't change under linear automorphisms of these spaces). Therefore the JK invariants of $R_{x+\lambda a}$ and $R_{\rho(g) (x +\lambda a)}$ coincide. Proposition~\ref{Prop:GroupPresJK} is proved.  \end{proof}

\subsection{Interpretation of Jordan-Kronecker invariants} \label{S:InterpretJK}

In this section we give an interpretation for some of the Jordan-Kronecker invariants. All properties we discuss here are quite elementary but will be useful in the sequel. We begin with the Jordan part, then discuss horizontal and vertical indices.

\subsubsection{Jordan blocks} 

The Jordan blocks correspond to \textit{elementary divisors} (see \cite{Gantmaher88}), which are defined using minors of a pencil $R_{x + \lambda a}$. For Lie algebra representations the Jordan blocks are connected with the singular set $\Sing$. In this section we connect the gcd of the highest order non-trivial minors  (i.e. the characterictic polynomial) with the singular set $\Sing$ of the representation.

\begin{definition}
The \textbf{rank of a pencil} $\mathcal P=\{A+\lambda B\}$ is the number \[\rk \mathcal P = \max_{\lambda\in \mathbb C} \rk (A+\lambda B).\] 
\end{definition}
\begin{definition}
The \textbf{characteristic polynomial}  $\charp(\alpha, \beta)$ of the pencil $\mathcal P$ is defined as the greatest common divisor of all  the $r\times r$ minors of the matrix  $\alpha A + \beta B$,  where $r=\rk \mathcal P$.
\end{definition} One can show that the polynomial $\charp(\alpha, \beta)$ does not depend on the choice of bases and therefore is an invariant of the pencil (up to multiplication on a constant).  It is easy to calculate $\charp(\alpha, \beta)$ for a single block from the JK  normal form:

\begin{itemize}

\item $\charp(\alpha, \beta) = 1$ for a horizontal or vertical Kronecker block;

\item $\charp(\alpha, \beta) = (\alpha \lambda_0 + \beta)^n$ for a $n \times n$ Jordan block with eigenvalue $\lambda_0$  (if $\lambda_0 = \infty$, then $\charp(\alpha, \beta) = \alpha^n$.) 

\end{itemize}

In general, $\charp(\alpha, \beta)$ is the product of characteristic polynomials for Jordan and Kronecker blocks (see \cite{Gantmaher88}). Thus $\charp(\alpha, \beta)$  is a homogeneous polynomial of degree equal to the sum of the sizes of all Jordan blocks. The characteristic polynomials of Jordan blocks $(\alpha \lambda_0 + \beta)^n$ are called {\it elementary divisors} of the pencil. For $\lambda_0 < \infty$ they can also be written as $(\lambda + \lambda_0)^n$ if we put $\lambda = \frac{\beta}{\alpha}$. The elementary divisors admit a natural invariant interpretation, see \cite{Gantmaher88} for details.

\begin{proposition}
\label{cor2}

The eigenvalues of Jordan blocks can be characterized as those  $\lambda\in\mathbb \C$ for which the rank of  $A-\lambda B$ drops,  i.e.  $\rk (A-\lambda B) < r=\rk \mathcal P$. The infinite eigenvalue appears in the case when $\rk B <r$. In other words, the eigenvalues of Jordan blocks, written as $-\lambda = (\beta : \alpha) \in \CP^1$, are solutions of the characteristic equation $\charp(\alpha, \beta)=0$. Moreover, multiplicity of each eigenvalue coincides with the multiplicity of the corresponding root of the characteristic equation. Jordan blocks are absent if and only if the rank of all non-trivial linear combinations $\alpha A + \beta B$ is the same.

\end{proposition}

\begin{definition}
A point $a \in V$ is called \textbf{regular}, if
\[
\dim \St_a \leq \dim \St_x \quad \text{for all } x \in V.
\]
Those points which are not regular are called \textbf{singular}.\end{definition}

Note that a \textit{stabilizer} of $x \in V$  can be defined in terms of $R_x$:
\[
\St_x = \Ker R_x = \{\xi\in\goth g~|~ R_x(\xi)=\rho(\xi) x =0 \}\subset \goth g.
\]

 The set of singular points will be denoted by $\Sing \subset V$. In terms of  $R_x$ we have
$$
\Sing = \{ y\in V~|~ \rk R_y < r= \max_{x\in V} \rk R_x\}.
$$

\begin{proposition}[\cite{BolsIzosKozl19}] \label{Prop:EigenSing}
\begin{enumerate}
\item The eigenvalues of Jordan blocks of a pencil $R_{a+\lambda b}$ are those values of  $\lambda\in\mathbb C$ for which the line $a- \lambda b$ intersects the singular set $\Sing$. In particular, $R_{a+\lambda b}$ has a Jordan block with eigenvalue $\infty$ if and only if $b$ is singular.
\item A generic pencil $R_a + \lambda R_b$ has no Jordan blocks if and only if the codimension of the singular set  $\Sing$ is greater or equal than  $2$.
\end{enumerate}
\end{proposition}

 Let us discuss the case $\codim\Sing =1$  in more detail. Consider the matrix of the operator $R_x$ and take all of its minors of size $r\times r$, where $r=\dimO$ is the dimension of a regular orbit, that do not vanish identically (such minors certainly exist). We consider them as polynomials  $p_1(x),\dots, p_N(x)$ on $V$. The singular set  $\Sing\subset V$ is then given by the system of polynomial equations
$$
p_i(x)=0,  \quad i=1,\dots, N.
$$

This set is of codimension one if and only if these polynomials possess a non-trivial greatest common divisor which we denote by $\mathsf{p}_\rho$.

Thus, we have $p_i(x) = \mathsf{p}_\rho (x) h_i(x)$,   which implies that the singular set  $\Sing$ can be represented as the union of two subsets:
\begin{equation}\label{sing01}
\mathsf{Sing}_0 =\{ \mathsf{p}_\rho (x) =0\} \quad \mbox{and} \quad \mathsf{Sing}_1 =\{ h_i (x) =0, \  i=1,\dots, N\}.
\end{equation}

It is easy to see that  $\mathsf{p}_\rho (x)$ is a semi-invariant of the representation $\rho$.  This follows from the fact that the action of $G$ leaves the singular set $\mathsf{Sing}_0$ invariant and therefore may only multiply $\mathsf{p}_\rho$ by a character of $G$.  We will refer to this polynomial $\mathsf{p}_\rho$ as the {\it fundamental semi-invariant} of  $\rho$. The fundamental semi-invariant is closely related to the characteristic polynomial  $\charp_{a,b}$ of the pencil $R_{a+\lambda b}$:

\begin{proposition}[\cite{BolsIzosKozl19}]\label{prop:fundinv}
Let $a,b\in V$ be such that the projective line $\alpha a+ \beta b$ does not intersect $\Sing_1$ and is not completely contained in $\Sing$ (i.e. contains at least one regular element). Then
$$
\charp_{a,b}(\alpha, \beta) = \mathsf{p}_\rho(\alpha a+ \beta b).
$$
\end{proposition}

\begin{corollary}[\cite{BolsIzosKozl19}]
\label{aboutsjord}
The degree of the fundamental semi-invariant $\mathsf{p}_\rho$ is equal to the sum of the sizes of all Jordan blocks for a generic pencil $R_{a+\lambda b}$. 
\end{corollary}

\begin{definition}
The total number of columns in horizontal blocks
$\mathrm{h}_{tot} = \sum_{i=1}^{\nh} \mathrm{h}_i$ is said to be  the \textbf{total Kronecker $h$-index} of the pencil $\mathcal P$. Similarly,  the total number of rows in the vertical Kronecker blocks  $\mathrm{v}_{tot} = \sum_{i=1}^{\nv} \mathrm{v}_j$ is said to be the \textbf{total Kronecker $v$-index} of $\mathcal P$.
\end{definition}

There is a relation between the total indices and the degree of the characteristic polynomial:

\begin{proposition}[\cite{BolsIzosKozl19}] We have
\begin{equation}
\label{sumofsum}
\mathrm{v}_{tot} + \mathrm{h}_{tot} = \dim V + \dim U - \rk \mathcal{P} - \deg \charp.
\end{equation}
\end{proposition}

Below we will also need the following simple statement about the number of Jordan blocks for each eigenvalue.

\begin{proposition} \label{Prop:JordBlock_EigenNumber} Consider an arbitrary representation of a Lie algebra $\rho: \mathfrak{g} \to \operatorname{gl}(V)$ and an arbitrary pencil $R_{x + \lambda a}$. Then the number of Jordan blocks with eigenvalue $\lambda_0$ is \[ \dim \Ker R_{x-\lambda_0 a} - \min_{\lambda} \dim \Ker R_{x+\lambda a} = \dim \St_{x-\lambda_0 a} - \dimSt. \] 

\end{proposition}

\subsubsection{Kronecker blocks} 

In general, Kronecker blocks for a pencil $R_x +\lambda R_a$ correspond to the polynomial solutions of the following equations~\eqref{onceagain} and \eqref{onceagaindual}.

\begin{proposition}[\cite{BolsIzosKozl19}]\label{prop:kerpencil}
Let $A$ be regular in a pencil $\mathcal P=\{ A+\lambda B\}$, i.e. $\rk A = \rk \mathcal P$. Then for every $u_0\in \Ker A$ there exists a sequence of vectors  $\{u_0, \dots, u_l \in U\}$ such that the expression $u(\lambda)=\sum_{j=0}^{l} u_j
\lambda^j$ is a solution of the equation
\begin{equation}
\label{onceagain}
(A+ \lambda B) u(\lambda)=0.
\end{equation}
Similarly, for any $u_0\in \Ker A^*$ there exists a sequence of vectors  $\{u_0, \dots, u_l \in V^*\}$ such that the expression $u(\lambda)=\sum_{j=0}^{l} u_j
\lambda^j$ is a solution of the equation
\begin{equation}
\label{onceagaindual}
(A+ \lambda B)^* u(\lambda)=0.
\end{equation}

\end{proposition}

As it was shown in \cite{BolsIzosKozl19}, if we consider polynomial solutions of $u_i(\lambda)$ of \eqref{onceagain} then \begin{equation}\label{firstDegreeIneq}\deg u_i \geq \mathrm{h}_i - 1,\end{equation}  where $ \deg u_{1}(\lambda) \leq \dots \leq \deg u_l(\lambda)$ and $\mathrm{h}_1 \leq \mathrm{h}_2 \leq \dots$ is the ordered sequence of horizontal indices. A similar statement holds for vertical indices and solutions of \eqref{onceagaindual}. It is not hard to find polynomial solutions for which the inequality \eqref{firstDegreeIneq} becomes an equality.

\begin{proposition}[\cite{BolsIzosKozl19}]
\label{L:ChainsRestrCor}
Horizontal Kronecker indices $\mathrm{h}_1, \dots, \mathrm{h}_p$ are given by $\mathrm{h}_i = r_i + 1$, where $r_1, \dots, r_p$ are the {minimal} degrees of independent solutions of \eqref{onceagain}. Similarly, vertical Kronecker indices $\mathrm{v}_1, \dots, \mathrm{v}_q$ are given by $\mathrm{v}_i = r'_i + 1$, where $r'_1, \dots, r'_q$ are the {minimal} degrees of independent solutions of the dual problem \eqref{onceagaindual}.
\end{proposition}

It is not hard to describe all solutions of \eqref{onceagain}. For Jordan and vertical Kronecker blocks there are no solutions of \eqref{onceagain}. Let $A, B: U \to V$ and let $f_0, \dots, f_{h-1}$ be a basis in $U$ from the JK normal form for a horizontal Kronecker block, i.e. \[ Af_0 = 0, \qquad Af_{1} = Bf_0, \qquad \dots,  \qquad Af_{h-1} = B f_{h-2}, \qquad Bf_{h-1} = 0.\] Then we have the following solution of \eqref{onceagain}: \[m(\lambda) = f_0 -\lambda f_1 + \dots + (-\lambda)^{h-1} f_{h-1}.\] We call $m(\lambda)$ a \textbf{minimal polynomial} corresponding to a horizontal Kronecker block. Obviously, these minimal polynomials are generators of polynomial solutions of \eqref{onceagain}
(as a $\mathbb{C} [\lambda]$-module).

\begin{proposition} \label{Prop:MinPolGenerators} 
Let $m_i(\lambda)$ be minimal polynomials of horizontal Kronecker blocks for any JK normal form of $A, B$. Then any polynomial solutions of  \eqref{onceagain} has the form \[ u(\lambda) = \sum_i P_i(\lambda) m_i(\lambda), \] where $P_i (\lambda) \in \mathbb{C}[\lambda]$ are arbitrary polynomials.
\end{proposition}

Below we use the following statement that allows us to calculate Kronecker indices. 

\begin{proposition}
\label{L:NumberHorizontInd}
Let $u_i(\lambda) = \sum_{j=0}^{r_i} u_{ij} \lambda^j$, for $i=1, \dots, p$, be polynomial solutions of  \eqref{onceagain}. If the following two conditions are satisfied:

\begin{enumerate}

\item the number of solutions $p$ is equal to the number of horizontal Kronecker blocks,

\item the vectors $u_{ij}$ are linearly independant,

\end{enumerate}

then the horizontal indices $\mathrm{h}_i = r_i + 1$. A similar statement holds for vertical Kronecker indices and solutions of \eqref{onceagaindual}.
\end{proposition}

\begin{proof}[Proof of Proposition~\ref{L:NumberHorizontInd}]  By Proposition~\ref{Prop:MinPolGenerators} all vectors $u_{ij}$ belong to the sum of horizontal Kronecker blocks. Since $u_{ij}$ are linearly independant, $\sum h_i \geq \deg u_i + 1$.  On the other hand, we have \eqref{firstDegreeIneq}. Thus, $h_i = \deg u_i + 1 = r_i + 1$. Proposition~\ref{L:NumberHorizontInd} is proved.  \end{proof}

For a Lie algebra representation $\rho: \mathfrak{g} \to \gl(V)$ the vertical indices are related with invariant polynomials of representations and horizontal indices are connected with stabilizers. 

The dimension of the stabilizer of a regular point is a natural characteristic of $\rho$ and is denoted by $\dimSt$.  It is convenient to keep in mind the action of the Lie group  $G$ associated with the Lie algebra $\goth g$ and its orbits. In particular,  for the dimension of a regular orbit we will use the notation $\dimO$. Notice that
$$
T_x\mathcal O_x = \mathrm{Im}\, R_x \quad \mbox{and} \quad \dim \mathcal O_x = \rk R_x.
$$

\begin{proposition}[\cite{BolsIzosKozl19}]\label{nhnv} For any Lie algebra representation $\rho$:
\begin{enumerate} \item the number  $\nh(\rho)$ of horizontal indices of  $\rho$  is equal to $\dimSt$.

\item the number  $\nv(\rho)$ of vertical indices of  $\rho$ is equal to  $\codimO$.
\end{enumerate}
\end{proposition}

It is not hard to prove the following statement about the number of indices equal to $1$.

\begin{proposition}\label{Prop:Triv_KroneckerBlocks_TotalNumber} For any pencil $R_x + \lambda R_a$ of a Lie algebra representation

\begin{enumerate}

\item the number of horizontal indices equal to $1$ (i.e. the number of horizontal $0 \times 1$ Kronecker blocks) is  \[\dim \left( \operatorname{St}_x \cap \operatorname{St}_a\right),\]

\item the number of vertical indices equal to $1$ (i.e. the number of vertical $1 \times 0$ Kronecker blocks) is \[\codim \left(\operatorname{Im} R_x + \operatorname{Im} R_a \right) = \dim \left( \Ker R_x^* \cap \Ker R_a^* \right).\]

\end{enumerate}

 \end{proposition} 

Since $\dim \left(\operatorname{Im} R_x + \operatorname{Im} R_a \right) \leq 2 \dim \mathfrak{g}$ there will always be trivial vertial blocks for high-dimensional representations. Namely, the following holds. 

\begin{corollary} 
\label{Cor:JKInv_BigRepr}
Consider an arbitrary representation of a Lie algebra $ \mathfrak{g} \to \gl(V).$ If $\dim V > 2 \dim \mathfrak{g}$, then in the JK decomposition there is no less than $(\dim V -2 \dim \mathfrak{g})$ vertical indices $\mathrm{v}_i = 1$. \end{corollary}

\subsubsection{Degrees of invariant polynomials and vertical indices} \label{S:DegInvPolin} 

Let $\rho \colon \g \to \gl(V)$ be a representation of a finite-dimensional Lie algebra $\g$ on a finite-dimensional vector space $V$. Let also  $\mathrm{v}_1(\rho), \dots, \mathrm{v}_q(\rho)$ be the vertical indices of  $\rho$. In this section we state three results from \cite{BolsIzosKozl19} about vertical indices and invariant polynomials. In sections below we use Theorem~\ref{thm3}.  Theorem~\ref{thm1} can also be useful in practice. Theorem~\ref{thm2} complements Theorem~\ref{thm3}.

The first result gives a bound for degrees of invariant polynomials in terms of vertical indices.  In the case of the coadjoint representation it was obtained by A.~Vorontsov \cite{Voron}.

\begin{theorem}[Lower bounds for degrees of polynomial invariants, \cite{BolsIzosKozl19}]\label{thm1}
\label{T:SumDeg_Polyn} Assume that $f_1, \dots, f_m$ are algebraically independent invariant polynomials of $\rho$, 
and $\deg f_1 \leq \dots \leq \deg f_m$. Then 
\begin{equation}
\label{Vorontsovestim}
\deg f_i \geq \mathrm{v}_i(\rho)
\end{equation}
for $i =1, \dots, m$.
\end{theorem}

\begin{corollary}[\cite{BolsIzosKozl19}]\label{sumvsind}
Suppose that there exist algebraically independent invariant polynomials  $f_1, f_2, \dots, f_q$,  $q=\codimO$, of a representation $\rho$ satisfying the condition
\begin{equation}\label{sumCond}
\sum_{i=1}^q \deg f_i  =  \svert (\rho). 
\end{equation}
Then
 \begin{equation}\label{sumCond2}
\mathrm{v}_i (\rho) =  \deg f_i.
\end{equation}
\end{corollary}

There are also two theorems about the case $\operatorname{deg} f_i = v_i(\rho)$.

\begin{theorem}[On the set where the invariants become dependent, \cite{BolsIzosKozl19}]\label{thm2}
Assume that $f_1, f_2, \dots, f_q$  is a complete set of algebraically independent invariant polynomials of $\rho$ (in other words, $q=\codimO$). The following conditions are equivalent:
 
 \begin{enumerate}
 \item The degrees of $f_i$'s are equal to the vertical indices of $\rho$: $\deg f_i = \mathrm{v}_i(\rho)$.
 \item The sum of the degrees of $f_i$'s is equal to the total vertical index of $\rho$: $\sum \deg f_i = \sum \mathrm{v}_i(\rho)$.
 \item The set where the differentials $df_1, \dots , df_q$ are linearly dependent has codimension $\geq 2$ in $V$.
 \item The set where the differentials $df_1, \dots , df_q$ are linearly dependent is contained in the set $\Sing_1$, i.e. in the codimension  $\geq 2$ stratum of the set of singular points of $\rho$ in $V$.
 \end{enumerate}

\end{theorem}

In practice, the following theorem can be even more useful.

\begin{theorem}[On polynomiality of the algebra of invariants, \cite{BolsIzosKozl19}]\label{thm3}
Assume that $f_1, f_2, \dots, f_q$  is a complete set of algebraically independent invariant polynomials of $\rho$ (in other words, $q=\codimO$). Then we have the following:
 \begin{enumerate}
 \item If the degrees of $f_i$'s are equal to the vertical indices: $\deg f_i = v_i(\rho)$ (equivalently, $\sum \deg f_i = \sum v_i(\rho)$), then the algebra $\mathbb C[V]^{\goth g}$ of polynomial invariants of $\rho$ is freely generated by $f_1, f_2, \dots, f_q$ (i.e. it is a polynomial algebra).
 \item Conversely, if the algebra $\mathbb C[V]^{\goth g}$ of polynomial invariants of $\rho$ is freely generated by $f_1, f_2, \dots, f_q$, and, in addition, $\rho$ has no proper semi-invariants (i.e. any semi-invariant is an invariant), then the degrees of $f_i$'s are equal to the vertical indices.
 
 \end{enumerate}

\end{theorem}

\begin{remark} A similar result is obtained in \cite{Joseph10} (see Corollary~$5.5$): if the algebra of
polynomial invariants is freely generated by $f_1, f_2, \dots, f_q$, and, in addition, a certain
semi-invariant is an invariant, then $\sum \deg g_i$ is equal to the degree of a certain form. \end{remark}

\section{JK invariants for standard representations of matrix
Lie algebras}
\label{S:SumsStandRepr}

\subsection{Standard representations of $\operatorname{gl}(n)$}
\label{SubS:GLStandRepr}

Consider the sum of $m$ standard representations of $\operatorname{gl}(n)$ \[\rho^{\oplus_m}: \operatorname{gl}(V) \to \operatorname{gl}(V^{\oplus_m}) \] Fix a basis of $V$, so that we could identify all the linear maps with matrices. Then an element of $V^{\oplus_m}$ is given by a matrix $X \in \operatorname{Mat}_{n \times m}$ and the corresponding linear mapping is the right multiplication: \begin{equation} \label{Eq:gl_n_StandReprSum_MatrixForm} \begin{gathered}  R_X: \operatorname{Mat}_{n \times n} \to \operatorname{Mat}_{n \times m} \\ Y \to YX  \end{gathered} \end{equation} In fact, in this case of $\operatorname{gl}(n)$ we can immediately describe the JK invaraints for any pair $R_X, R_A$. Moreover, we can easily do it for an arbitrary composition of matrices, not necessarily square ones.

\begin{theorem} \label{T:JK_RightMatrixAction} Consider a pencil of right multiplications  \begin{align} \label{Eq:RightMult} R_X + \lambda R_A: & \operatorname{Mat}_{k \times n} \to \operatorname{Mat}_{k \times m} \\ & Y \to Y (X+\lambda A)\end{align}  for arbitrary $X, A \in \operatorname{Mat}_{n \times m}$. Then the JK decomposition of the pencil $R_x + \lambda A$ is obtained from the JK decomposition of the pencil $X^T + \lambda A^T$ if we take each block $k$ times. \end{theorem}

First, let us prove that  the natural left-right $\operatorname{GL}(n)\times \operatorname{GL}(m)$  action doesn't change JK decompositions.

\begin{proposition}  \label{Prop:GL_SumRepr_LeftRightAction} Consider the pencil of right multiplications \eqref{Eq:RightMult}. Then for any $C \in \operatorname{GL}(n), D \in \operatorname{GL}(m)$ the JK decompositions of the pencil $R_X + \lambda R_A$ and for the pencil $R_{CXD} + \lambda R_{CAD}$ coincide. \end{proposition}

\begin{proof}[Proof of Proposition~\ref{Prop:GL_SumRepr_LeftRightAction}] JK decomposition for any pair of linear maps doesn't depend on a choice of bases in the spaces. Thus the JK decomposition of $R_X + \lambda R_A$ won't change if we take their composition with any linear automorphisms $\varphi, \psi$ of $\operatorname{Mat}_{k\times n}$ and $\operatorname{Mat}_{k\times m}$. It remains to notice that for the automophisms \[\varphi(Y) = YC, \qquad \psi(Z) = ZD\] the following diagram commutes \[\begin{CD}
\operatorname{Mat}_{k \times n}  @>R_{CXD} + \lambda R_{CAD} >> \operatorname{Mat}_{n \times m} \\
@VV \varphi V @AA\psi A\\
\operatorname{Mat}_{k \times n}  @>R_X + \lambda R_A>> \operatorname{Mat}_{n \times m}
\end{CD}
\]  Proposition~\ref{Prop:GL_SumRepr_LeftRightAction} is proved. \end{proof}

\begin{proof}[Proof of Theorem~\ref{T:JK_RightMatrixAction}] $X$ and $A$ multiply each row of $Y \in \operatorname{Mat}_{k \times n}$ independantly. Thus if we regard ${k \times n}$ matrices as vector-columns consisting of its $1 \times n$ rows, then the operators  $R_X$ and $R_A$ are given by block-diagonal matrices \begin{equation} \label{Eq:MatrixMultByRows}  R_X = \left( \begin{matrix} X^T & & \\ & \ddots & \\ & & X^T\end{matrix} \right), \qquad R_A = \left( \begin{matrix} A^T & & \\ & \ddots & \\ & & A^T\end{matrix} \right) \end{equation}  with $k$ blocks. Here the matrices are transposed because we act right on $k \times n$ matrices and left on vectors. Using Proposition~\ref{Prop:GL_SumRepr_LeftRightAction} we can assume that $X^T$ and $A^T$ are as in the JK theorem \ref{T:JK_operator}. Then $R_X$ and $R_A$ are also as in the JK theorem \ref{T:JK_operator}. Theorem~\ref{T:JK_RightMatrixAction} is proved. \end{proof}

We have found the JK decomposition for any pair of matrices $X, A$, but what is the generic one?  In the following Lemma, answering this question, we denote by $I_k$ the $k\times k$ identity matrix.

\begin{lemma} \label{L:GenMatPair} Consider the sum of $m$ standard representations for the Lie algebra $\operatorname{gl}(n)$.  Define $n \times m$ matrices $X$ and $A$ as follows: \begin{itemize}

\item If $m<n$: \begin{equation} \label{Eq:GenMatPairRowMoreCol}  X = \left( \begin{matrix} I_m \\ 0 \end{matrix} \right), \qquad A = \left( \begin{matrix} 0 \\  I_m \end{matrix} \right).\end{equation}

\item If $n=m$: \begin{equation} \label{Eq:GenMatPairRowEqCol}   X = \left( \begin{matrix} \lambda_1 & & \\ & \ddots & \\ & & \lambda_n \end{matrix} \right), \qquad A = \left( \begin{matrix} 1 & & \\ & \ddots & \\ & & 1 \end{matrix} \right), \end{equation} where $\lambda_i$ are distinct numbers.

\item If $m>n$ \begin{equation} \label{Eq:GenMatPairRowLessCol}    X = \left( \begin{matrix} I_n & 0 \end{matrix} \right), \qquad A = \left( \begin{matrix} 0 & I_n \end{matrix} \right).\end{equation}

\end{itemize}

Then the pencil $R_{X +\lambda A}$ is generic. \end{lemma}

\begin{proof}[Proof of Lemma \ref{L:GenMatPair}] By Proposition~\ref{Prop:GL_SumRepr_LeftRightAction}, it suffices to show that there exists an open dense subset $U \subset \mathrm{Mat}_{n \times m} \times \mathrm{Mat}_{n \times m}$ such that for any pair $\left( \tilde{X}, \tilde{A} \right) \in U$ its orbit under the natural left-right action of $\operatorname{GL}(n) \times \operatorname{GL}(m)$ on $\operatorname{Mat}_{n \times m}$ contains a pair $(X, A)$ of the above form. This is a simple and well-known fact, see e.g. \cite{Pokrzywa86}[p. 119]. Lemma \ref{L:GenMatPair} is proved. \end{proof}

\begin{theorem} \label{T:JKSumStandardGLn} Let $\rho$ be the sum of m standard representations of $\operatorname{gl}(n)$.

\begin{enumerate}

\item Assume that $m<n$. Divide $m$ by $n-m$ with the remainder: \begin{equation} \label{Eq:Divide_MBy_NMinM} m = q (n-m) + r, \qquad q, r \in \mathbb{Z}, \qquad 0 \leq r < n-m.\end{equation} Then the JK invariants of $\rho$ are $n(n-m)$ horizontal indices: \[ \underbrace{q+1, \dots, q+1}_{n(n-m-r)}, \qquad \underbrace{q+2, \dots, q+2}_{nr}.\]

\item Assume that $n=m$. The JK invariants of $\rho$ consist of $n$ distinct eigenvalues with $n$ Jordan $1 \times 1$ blocks corresponding to each eigenvalue.

\item Assume that $m>n$.  Divide $n$ by $m-n$ with the remainder:  \[ n = q (m-n) + r, \qquad q, r \in \mathbb{Z}, \qquad 0 \leq r < (m-n).\] Then the JK invariants of $\rho$ are $n(m-n)$ vertical indices: \[ \underbrace{q+1, \dots, q+1}_{n(m-n-r)}, \qquad \underbrace{q+2, \dots, q+2}_{nr}.\]

\end{enumerate}

\end{theorem} 

 By Theorem~\ref{T:JK_RightMatrixAction} the JK invariants of $\rho$ are the JK invariants of a generic pair of $(m \times n)$-matrices taken $n$ times. Thus, we get the following.

\begin{corollary} \label{Cor:GenPenJK} A generic pair $A, B \in \operatorname{Mat}_{m \times n}$ has the following JK invariants: 

\begin{enumerate}

\item If $m < n$ and $m = q(n-m) +r$, then the JK invariants are $(n-m-r)$ horizontal indices $h_i = q+1$ and $r$ horizontal indices $h_i = q+2$. 

\item If $m = n$, then the JK invariants are  $n$ Jordan $1\times 1$ blocks with distinct eigenvalues.

\item If $m > n$ and $m = q(m-n) +r$, then the JK invariants are $(m-n-r)$ vertical indices $h_i = q+1$ and $r$ vertical indices $h_i = q+2$. 

\end{enumerate}

\end{corollary}

\begin{remark} \label{Rem:SumBlocksGL} For any JK decomposition we can always check that the sum of sizes of blocks is equal to the size of matrices. The sizes of blocks are: \begin{itemize}

\item $(\mathrm{h}_i -1) \times \mathrm{h}_i$ for horizontal indices;

\item $\mathrm{v}_i \times (\mathrm{v}_i-1)$ for vertical indices;

\item $n_i \times n_i$ for Jordan blocks.

\end{itemize} 

For a Lie algebra representation $\rho: \mathfrak{g} \to  V$ the corresponding matrices have the size $\operatorname{dim} V \times \operatorname{dim} \mathfrak{g}$. Hence, for the sum of m standard representations of $\operatorname{gl}(n)$ the total size is $nm \times n^2$. For $n < m$ we can informally write the equality for the sum of sizes as \[n(n-m-r) \left[q \times (q+1) \right] + nr \left[ (q+1) \times (q+2)\right] = nm \times n^2.\] For $n=m$ we get an obivous ``equality'' \[ n^2 (1 \times 1) = n^2 \times n^2.\] And for $n >m $ we get a similar informal equality:  \[n(m-n-r) \left[(q+1) \times q \right] + nr \left[ (q+2) \times (q+1)\right] = nm \times n^2.\] In these equalities $x \times y$ is just a notation for a vector $(x, y) \in \mathbb{Z}^2$. 
\end{remark}

\begin{remark} Note that for $m \not =n$ in Theorem~\ref{T:JKSumStandardGLn} the Kronecker indices are ``as equal as possible'': they don't differ more then by $1$. \end{remark}

\begin{proof}[Proof of Theorem~\ref{T:JKSumStandardGLn}]  The proof is rather straightforward: we take a generic pair $(X, A)$ and then explicitly describe bases for all the blocks in the JK decomposition. Of course, we try to find a generic pair $(X, A)$ with a ``simple form'', to simplify the calculations. Luckily, we already know a ``rather simple'' generic pair $(X, A)$ from Lemma~\ref{L:GenMatPair}.  Now we describe the bases of the blocks for this pair $(X, A)$. Denote by $E_{k, l}$ the matrix with $1$ at position $(k,l)$ and zeros everywhere else. 

\begin{itemize}

\item First, consider the case $m<n$. For the pair \eqref{Eq:GenMatPairRowMoreCol} the following identity holds \begin{equation}  \label{Eq:ChainEq_MLessN}  E_{i,j }  X = E_{i, j+(n-m)} A \end{equation}  Here if the second index goes beyond the interval $[1, n]$, then we formally put that side to zero. Thus for any $i, j$ such that \[1 \leq i \leq n, \quad 1 \leq j \leq n-m\] we get a horizontal Kronecker block with the bases \begin{equation} \label{Eq:GLReprSum_KroneckerBasis_MLessN} E_{i, j+ l_1 (m-n)}, \quad  \text{ in } \operatorname{gl}(n), \qquad \text{ and } \qquad E_{i, j+ l_2(n-m)} \quad \text{ in } \quad V^{\oplus_m}.\end{equation} We take all $l_1$ and $l_2$ such that the matrices lie in the corresponding spaces.
We get the horizontal indices $\mathrm{h}_i = q+2$ for  $1 \leq j \leq r$ and $\mathrm{h}_i = q+1$ for $r+1 \leq j \leq n-m$.

\item Then, in the case $m=n$ we have obvious Jordan blocks defined by the following idenities for the pairs \eqref{Eq:GenMatPairRowEqCol} \begin{equation} \label{Eq:GLReprSum_JordanBasis} E_{ij }  X = \lambda_j E_{ij} A\end{equation} 

\item Finally, we have a natural duality between the cases $m>n$ and $n<m$ because the pair  \eqref{Eq:GenMatPairRowLessCol}  is the transpose of the pair \eqref{Eq:GenMatPairRowMoreCol}. 

\end{itemize}

Theorem~\ref{T:JKSumStandardGLn} is proved. \end{proof}

\begin{remark}The algebra $\mathbb{C}[V]^{\mathfrak{\rho}}$ of polynomial invariants for the sum of standard representations of $\operatorname{gl}(n)$  is trivial. Unlike the case of $\operatorname{sl}(n)$ below (see Proposition~\ref{Prop:SLn}), the invariants are quotients of minors, which are rational functions, not polynomials. Hence we cannot apply Theorem~\ref{thm3}. \end{remark}

\subsection{Standard representations of $\operatorname{sl}(n)$}
\label{SubS:SLStandRepr}

\subsubsection{JK invariants and explicit bases}

Let us now study the sum of standard representations of $\operatorname{sl}(n)$, which we regard as the space of matrices with zero traces.

\begin{theorem}[\cite{BolsIzosKozl19}]\label{T:JKSumStandardSLn} Let $\rho$ be the sum of $m$ standard representations of $\operatorname{sl}(n)$.
\begin{enumerate}
\item Assume that $m<n$. Divide $m$ by $n-m$ with the remainder: \[ m = q (n-m) + r, \qquad q, r \in \mathbb{Z}, \qquad 0 \leq r < n-m.\] Then the JK invariants of $\rho$ consist of $n(n-m)-1$ horizontal indices: \[ \underbrace{q+1, \dots, q+1}_{n(n-m-r) -(q+1)}, \qquad \underbrace{q+2, \dots, q+2}_{nr+q}.\]

\item Assume that $m=n$. Then the JK invariants of $\rho$ consist of: 

\begin{itemize}

\item one vertical index $\mathrm{v}_1 =n$ 

\item $n$ distinct eigenvalues with $n-1$ Jordan $1 \times 1$ blocks corresponding to each eigenvalue.

\end{itemize}

\item Assume that $m>n$. Divide $n$ by $m-n$ with the remainder:  \[ n = q (m-n) + r, \qquad q, r \in \mathbb{Z}, \qquad 0 \leq r < (m-n).\] Then the JK invariants of $\rho$ consist of the following $n(m-n)+1$ vertical indices.

\begin{itemize}
 \item If $r \not = 0$, then the vertical indices are 
  \[ \underbrace{q+1, \dots, q+1}_{n(m-n-r) + (q+2) }, \qquad \underbrace{q+2, \dots, q+2}_{nr - (q+1)}.\]
 \item If $r = 0$, then the vertical indices are    \[  \underbrace{q, \dots, q}_{q+1} \qquad \underbrace{     q+1, \dots, q+1}_{n(m-n)-q}.\]
   \end{itemize}

\end{enumerate}

\end{theorem}

\begin{remark} Similar to Remark~\ref{Rem:SumBlocksGL} we can check that the sum of sizes of blocks is \[\operatorname{dim} V \times \operatorname{dim} \mathfrak{g} = nm \times (n^2 - 1).\] For $n < m$ we get an ``equality'' \begin{gather} \left( n(n-m-r) -(q+1) \right) \left[q \times (q+1) \right] + \left( nr + q\right) \left[ (q+1) \times (q+2)\right] = \\  =nm \times n^2 - (q+1)\left[q \times (q+1) \right] + q\left[ (q+1) \times (q+2)\right] = nm \times (n^2 -1).\end{gather} For $n=m$ we have \[ n \times (n-1) + n(n-1) (1 \times n) = n^2 \times (n^2-1).\] For $n >m $, $r\not  = 0$ we get \begin{gather} \left( n(m-n-r) + (q+2) \right) \left[(q+1) \times q \right] + \left( nr - (q+1) \right) \left[ (q+2) \times (q+1)\right] = \\ =  nm \times n^2 +  (q+2) \left[(q+1) \times q \right]  - (q+1)  \left[ (q+2) \times (q+1)\right] = nm \times \left( n^2 -1\right).\end{gather} And for $n > m$, $r =0$ since $n = q(n-m)$ we get \[(q+1) \left[ q \times (q-1) \right] + (n(m-n) -q) \left[ (q+1) \times q\right] =  nm \times \left( n^2 -1\right).\] All the necessary equalities hold.
\end{remark}

\begin{remark} Similar to the case of $\operatorname{gl}(n)$ in Theorem~\ref{T:JKSumStandardSLn}  the kronecker indices for $m \not = n$ don't differ more than by $1$ (are ``as equal as possible''). \end{remark}

First, we determine generic pairs $X, A$ for the Lie algebra  $\operatorname{sl}(n)$. Since below we will need a similar fact for $\operatorname{so}(n)$ and $\operatorname{sp}(n)$ we simultaneously prove this fact for all these Lie algebars.

\begin{lemma} \label{L:GenMatPairSL} The generic pencils $R_{X + \lambda A}$ for the sum of $m$ standard representations for any of the Lie algebras $\operatorname{sl}(n)$, $\operatorname{so}(n)$ or  $\operatorname{sp}(n)$ are the same as in Lemma~\ref{L:GenMatPair} for the sum of $m$ standard representations of $\operatorname{gl}(n)$.\end{lemma}

The proof of Lemma~\ref{L:GenMatPairSL} is the same as for the Lemma~\ref{L:GenMatPair}. All we need is to replace Proposition~\ref{Prop:GL_SumRepr_LeftRightAction} with the following similar fact.

\begin{proposition}  \label{Prop:SL_SumRepr_LeftRightAction} Consider the sum of $m$ standard representations for any of the Lie algebras $\operatorname{gl}(n)$, $\operatorname{sl}(n)$, $\operatorname{so}(n)$ or $\operatorname{sp}(n)$.  Then for any $C \in \operatorname{GL}(n)$ and $D \in \operatorname{GL}(m)$ the JK decompositions of a pencil $R_X +\lambda R_A$ and the pencil $R_{CXD} + \lambda R_{CAD}$ coincide.\end{proposition}

\begin{proof}[Proof of Proposition~\ref{Prop:GL_SumRepr_LeftRightAction}] 
It suffices to prove the existence of invertible linear mappings $\varphi: \mathfrak{g} \to \mathfrak{g}$ and $\psi: \operatorname{Mat}_{n \times m} \to \operatorname{Mat}_{n \times m}$ (where $\varphi$ does not have to be a Lie algebra automorphism) such that the following diagram commutes\[
\begin{CD}
\mathfrak{g}  @>R_{CXD} + \lambda R_{CAD} >> \operatorname{Mat}_{n \times m} \\
@VV \varphi V @AA\psi A\\
\mathfrak{g} @>R_X + \lambda R_A>> \operatorname{Mat}_{n \times m}
\end{CD}
\] For  $\operatorname{gl}(n)$ and $\operatorname{sl}(n)$, a suitable choice of $\varphi$, $\psi$ is \[\varphi(Y) = C^{-1}YC, \qquad \psi(Z) = C ZD. \] For $\operatorname{so}(n)$ and $\operatorname{sp}(n)$ we take \[\varphi(Y) = C^{*}YC, \qquad \psi(Z) = (C^{*})^{-1} ZD, \] where $C^*=C^T$ for $\operatorname{so}(n)$ and $C^*=\Omega^{-1}C^T\Omega$ for $\operatorname{sp}(n)$, with $\Omega$  being the matrix of the symplectic form given by formula \eqref{Eq:SympLieAlg_Elem} below.  In all cases we see that the automorphisms $\varphi$, $\psi$  intertwine the pencils  $R_{X + \lambda A}$ and  $R_{CXD + \lambda CAD}$, so those pencils have the same JK type, as stated. Proposition~\ref{Prop:SL_SumRepr_LeftRightAction} is proved. \end{proof}

\begin{proof}[Proof of Theorem~\ref{T:JKSumStandardSLn}] The proof is similar to the proof of Theorem~\ref{T:JKSumStandardGLn}. First, we take a generic take $(X, A)$.  As for $\gl(n)$ before, for $\operatorname{sl}(n)$ we take the same generic pair $(X, A)$ from Lemma~\ref{L:GenMatPair} (it is generic for $\operatorname{sl}(n)$ by Lemma~\ref{L:GenMatPairSL}). 

Now we describe the bases of blocks for this pair $(X, A)$. As we shall see below, the majority of bases of blocks for $\gl(n)$ remain the same for $\operatorname{sl}(n)$. We just have ``to fix'' the bases for $\gl(n)$ that contain elements that are not in $\operatorname{sl}(n)$. For each Kronecker block we describe only its basis $\Lambda_{1}, \dots, \Lambda_{\alpha}$ in $\operatorname{sl}(n)$. Since the blocks have to be as in the JK Theorem~\ref{T:JK_operator} the corresponding basis in $\operatorname{Mat}_{n \times m}$  for vertical Kronecker blocks  is \begin{equation} \label{Eq:KroneckerBlocks_ImageBasis} \Lambda_0 X, \qquad \Lambda_1 X = \Lambda_0 A, \qquad  \dots \qquad \Lambda_{\alpha} X = \Lambda_{\alpha-1} A, \qquad \Lambda_{\alpha} A \end{equation} and for horizontal Kronecker blocks we omit $\Lambda_0 X = \Lambda_{\alpha}A =0$ in the basis \eqref{Eq:KroneckerBlocks_ImageBasis}.

It is an easy exercise to check that the union of described vectors for all the blocks form bases of $\operatorname{sl}(n)$ and $\operatorname{Mat}_{n\times m}$.

\begin{enumerate}

\item Case $m <n$. The Kronecker blocks for $\operatorname{gl}(n)$ with bases \eqref{Eq:GLReprSum_KroneckerBasis_MLessN} that do not contain diagonal matrices $E_{i,i} \in\operatorname{gl}(n)$ 
remain unchanged for $\operatorname{sl}(n)$. We can replace the bases with $E_{i,i}$ by their linear combinations that contain elements $E_{i,i} - E_{i+1, i+1}$. In other words, for each $1 \leq i \leq n-1$ we take Kronecker blocks with bases \begin{equation} \label{Eq:SL_DiagKronBase} E_{i, i+ k(n-m)} - E_{i+1, i+1 +k(n-m)}\end{equation} in $\operatorname{sl}(n)$. Here we formally assume that the matrix is equal to zero if the second index doesn't belong to $[1, n]$. The  number $k \in \mathbb{Z}$ for \eqref{Eq:SL_DiagKronBase}  runs through all values for which \begin{equation} \label{Eq:SL_SumRempr_MLessN_CondK}  0 \leq i + k(n-m)  \leq n.  \end{equation} The horizontal index of the block is the number of values of $k$ that satisfy \eqref{Eq:SL_SumRempr_MLessN_CondK}. It is $q+1$ if \[ 0 \leq i \mod{(n-m)}  \leq r \] and is equal to $q$ otherwise. Since $1 \leq i \leq n-1$ and by \eqref{Eq:Divide_MBy_NMinM} we have \[n = (q+1)(n-m) +r,\] there are total $(q+2)(r+1)-2$ horizontal indices $q+2$, and $(q+1)(n-m-r-1)$ indices $q+1$ for the blocks \eqref{Eq:SL_DiagKronBase}. Thus $q+1$  horizontal indices $\mathrm{h}_i = q+1$ for $\operatorname{gl}(n)$ are replaced by $q$ horizontal indices $\mathrm{h}_i = q+2$ for $\operatorname{sl}(n)$.

\item Case $m=n$. The Jordan blocks defined by \eqref{Eq:GLReprSum_JordanBasis} for $i \not =j$ remain the same for $\operatorname{sl}(n)$. The remaining Jordan blocks given by \eqref{Eq:GLReprSum_JordanBasis} correspond to diagonal matrices: \begin{equation} \label{Eq:DiagBasis} E_{ii} X = \lambda_i E_{ii} A.\end{equation} We have to change them since $E_{ii} \not \in \operatorname{sl}(n)$. Let us show that these $n$ Jordan blocks for $\operatorname{gl}(n)$ are replaced with the following a $n \times (n-1)$ vertical Kronecker block for $\operatorname{sl}(n)$. Let $(\mu_1, \dots, \mu_n)$ be a vector orthogonal to the vectors \[ \left( \frac{1}{\lambda_1^j}, \dots, \frac{1}{\lambda_n^j} \right), \qquad j=0, \dots, n-2.\]   Then the basis of the $n \times (n-1)$ vertical Kronecker block in $\operatorname{sl}(n)$ is \[ \Lambda_j =  \operatorname{diag} \left( \frac{\mu_1}{\lambda_1^j}, \dots, \frac{\mu_n}{\lambda_n^j} \right), \qquad j=0, \dots, n-2\]  and the bases in $\operatorname{Mat}_{n\times n}$ is given by \eqref{Eq:KroneckerBlocks_ImageBasis}. It remains to check that the basis is correctly defined, i.e. that $\Lambda_j \in \operatorname{sl}(n)$ and that \begin{equation} \label{Eq:ChainBases}\Lambda_j X = \Lambda_{j-1} A.\end{equation} 

Indeed, $\Lambda_j \in \operatorname{sl}(n)$, since, by construction, $ \displaystyle \sum_i \frac{\mu_i}{\lambda_i^j} = 0$. And  \eqref{Eq:ChainBases} is satisfied, since, by \eqref{Eq:DiagBasis}, we have \[ \Lambda_j X = \sum_i \frac{\mu_i}{\lambda_i^j} E_{ii} X = \sum_i \frac{\mu_i}{\lambda_i^{j-1}} E_{ii} A = \Lambda_{j-1} A.\] 

\item Case $m>n$ can be considered similarly to the case $m<n$. For the pair \eqref{Eq:GenMatPairRowLessCol} we have the identity \[  E_{i,j +(m-n) }  X = E_{i, j} A \]  For any $1 \leq i \leq n, 1 \leq j \leq m-n$ such that the sequence \begin{equation} \label{Eq:SL_DiagKronBaseGener} E_{i, j+k(m-n)} \end{equation} does not contain $E_{i,i}$ we can take \eqref{Eq:SL_DiagKronBaseGener} as a basis  in $\operatorname{sl}(n)$ of a vertical Kronecker block. These $n(n-m-1)$ Kronecker blocks remain from the case of $\operatorname{gl}(n)$. Other ``new'' $n+1$ Kronecker blocks are as follows.

For each $1 \leq i \leq n-1$ we take the Kronecker block which basis in $\operatorname{sl}(n)$ is \begin{equation} \label{Eq:SL_DiagKronBaseMBig} E_{i, i+ k(m-n)} - E_{i+1, i+1 +k(m-n)}\end{equation} where $k \in \mathbb{Z}$ is bounded by \[ 1 \leq i + k(m-n)  \leq n-1. \] We take 2 more Kronecker blocks with the following bases in $\operatorname{sl}(n)$: \begin{equation} \label{Eq:SL_DiagKronMBigNew1} E_{1, 1 + k(m-n)}, \qquad 0< k(m-n) < n\end{equation} and 
\begin{equation} \label{Eq:SL_DiagKronMBigNewN} E_{n, n - k(m-n)}, \qquad 0 < k(m-n) <n. \end{equation}  Note that we take $k>0$ so that there is no elements $E_{1,1}$ and $E_{n+1, n+1}$.

It is not hard to find the vertical indices of blocks with bases \eqref{Eq:SL_DiagKronBaseMBig}, \eqref{Eq:SL_DiagKronMBigNew1} and \eqref{Eq:SL_DiagKronMBigNewN}, i.e. to find the number of possible $k$ in each case. If $r \not = 0$, then the vertical index of the block is $\mathrm{v}_j = q+2$ for \[1 \leq i \mod{m -n} \leq r-1\] and is equal to $q+1$ otherwise.  If $r = 0$, then the vertical index $\mathrm{v}_j = q+1$ for \[ i \mod{m -n} = 0\] and is equal to $q$ otherwise. For \eqref{Eq:SL_DiagKronMBigNew1} and \eqref{Eq:SL_DiagKronMBigNewN} both vertical indices are equal to $q+1$ for $r \not = 0$ and are equal to $q$ otherwise. 

Hence if $r\not =0$, then $q+1$ vertical indices $\mathrm{v}_i = q+2$ for $\operatorname{gl}(n)$ are replaced by $q+2$ vertical indices $\mathrm{v}_i = q+1$ for $\operatorname{sl}(n)$. And for $r=0$ instead of $q$ vertical indices $\mathrm{v}_i =q+1$ for $\operatorname{gl}(n)$  we get $q+1$ vertical indices $\mathrm{v}_i =q$ for $\operatorname{sl}(n)$.

\end{enumerate}

Theorem~\ref{T:JKSumStandardSLn} is proved. \end{proof}

\subsubsection{Number and types of blocks} \label{SubS:BlockTypesSln}

Often it is much simpler to find the number of blocks in the JK decomposition, then to construct bases for them. In these section we compute some simple invariants using propositions from Section~\ref{S:InterpretJK} and show that Theorem~\eqref{T:JKSumStandardSLn} can be quite easily proved in the cases $m < \frac{n}{2}$, $m=n$, $m=n+1$ and $m \geq 2n$.

First we can find the singular set $\Sing$ and the number of indices using Propositions~\ref{Prop:EigenSing} and \ref{nhnv}. We get the following trivial statement.

\begin{proposition} \label{Prop:SLn_SingStOreg} For the sum of $m$ standard representations of $\operatorname{sl}(n)$ we have  \[ \Sing = \left\{ X \in \operatorname{Mat}_{n \times m} | \, \, \rk X<\min{(m, n)} \right\}.\] Thus there are Jordan blocks only for $m=n$. The numbers of horizontal and vertical indices are \[ \dimSt = \begin{cases} n(n-m)-1, &\quad m <n, \\ 0, &\quad m\geq n, \end{cases} \qquad  \codimO = \begin{cases} 0, &\quad m<n, \\ n(m-n)+1, &\quad m \geq n, \end{cases} \]  respectively. \end{proposition} 

\begin{proof}[Proof of Proposition~\ref{Prop:SLn_SingStOreg}] Any element $X \in \operatorname{Mat}_{n \times m}$ by the left-right $\operatorname{GL}(n) \times \operatorname{GL}(m)$ action can be taken to the matrix \[ E_k = \sum_{i=1}^k E_{11} \] where $k \leq \min{(m, n)}$.  Then $\operatorname{St}_{E_k}$ are matrices in $\operatorname{sl}(n)$ with zeroes in the first $k$ columns. The rest easily follows. Proposition~\ref{Prop:SLn_SingStOreg} is proved. \end{proof}

Then we can compute  the number of indices equal to $1$ using Proposition~\ref{Prop:Triv_KroneckerBlocks_TotalNumber}.

\begin{corollary} 
\label{Cor:SL_SumStandRepr_TrivKronBlocks}
Consider a generic pencil $R_{X }+\lambda R_A$ for the sum of $m$ standard representations of $\operatorname{sl}(n)$. The number of horizontal indices equal to $1$  is
\begin{equation} \label{Eq:SL_SumStandRepr_TrivHorizKronNum} \dim \left(\operatorname{St}_X \cap \operatorname{St}_A\right) = \begin{cases} n(n-2m)-1 , &\qquad m< \frac{n}{2}, \\ 0, &\qquad m \geq  \frac{n}{2}.\end{cases}  \end{equation} The number of vertical indices equal to $1$ is  \begin{equation} \label{Eq:SL_SumStandRepr_TrivVertKronNum} \codim \left(\operatorname{Im} R_X + \operatorname{Im} R_A \right) =   \begin{cases}    0, &\qquad m< 2n, \\n(m-2n)+2, &\qquad m \geq 2n. \end{cases} \end{equation}\end{corollary}

\begin{proof}[Proof of Corollary~\ref{Cor:SL_SumStandRepr_TrivKronBlocks}] The proof is by direct calculation for pairs $(X, A)$ from Lemma~\ref{L:GenMatPair}. This is an easy exercise. For example, $\operatorname{St}_X \cap \operatorname{St}_A$ are matrices with first $k$ and last $k$ zero columns. And in the case $m\geq 2n$ for $(X, A)$ given by \eqref{Eq:GenMatPairRowLessCol} $\operatorname{Im} R_x + \operatorname{Im} R_a$  consists of matrices \[ \left( \begin{matrix}  Y_1 & 0 & Y_2 \end{matrix} \right), \qquad \operatorname{tr}Y_i =0, \] where $Y_i$ are $n\times n$ matrices. Corollary~\ref{Cor:SL_SumStandRepr_TrivKronBlocks} is proved. \end{proof}

From Proposition~\ref{Prop:SLn_SingStOreg}, Corollary~\ref{Cor:SL_SumStandRepr_TrivKronBlocks} and considerations of dimensions we get the following.

\begin{corollary}  For $m<  \frac{n}{2}$ and $m \geq 2n$ the JK invariants are as in Theorem~\ref{T:JKSumStandardSLn}. \end{corollary}

For $m=n$ we can determine the sizes of all blocks using Proposition~\ref{Prop:JordBlock_EigenNumber}.

\begin{proposition}  \label{Prop:SLMEqNorNplus1} For $m=n$ and $n+1$ the JK invariants are as in Theorem~\ref{T:JKSumStandardSLn}. \end{proposition}

\begin{proof} [Proof of Proposition~\ref{Prop:SLMEqNorNplus1}] The minors of order $n$ of $X$  form a complete set of algebraically independant invariant polinomials.  Their differentials are linearly dependant on the set of codimension $\geq 2$ and thus by Theorem~\ref{thm2} all vertical indices $v_i =n$. 

For the sum of $m=n$ standard representations of $\operatorname{sl}(n)$ we have \begin{equation} \Sing = \Sing_0=\left\{ \det X = 0 \right\} \end{equation} By Lemma~\ref{L:GenMatPair} for a generic pair $(X, A)$ there are $n$ distinct eigenvalues $\lambda_i$ and for each of them $\rk (X+\lambda_i A)= n-1$. Thus by  Proposition~\ref{Prop:JordBlock_EigenNumber} there are $n-1$ Jordan blocks corresponding to each eigenvalue. From considerations of dimensions they are all $1 \times 1$.  Proposition~\ref{Prop:SLMEqNorNplus1} is proved. \end{proof}

\subsubsection{Algebra of invariant polynomials}

Now, let us demonstrate how Theorem~\ref{thm3} works in this case. The following proposition and its proof are taken from \cite{BolsIzosKozl19} for the sake of completeness.

\begin{proposition} \label{Prop:SLn} Let $\rho : \operatorname{sl}(n) \to \operatorname{Mat}_{n \times m}$ be the sum of $m$ standard representations of  $\operatorname{sl}(n)$,  and let $X \in \operatorname{Mat}_{n \times m}$. Then: 

\begin{enumerate}

\item If $m<n$, the algebra of invariants of $\rho$ is trivial.

\item If $m=n$, the algebra of invariants is freely generated by the polynomial $\operatorname{det} X$. 

\item If $m=n+1$,  the algebra of invariants is freely generated by $n\times n$ minors of $X$.

\item If $m \geq 2n$, the algebra of invariants is not freely generated.

\end{enumerate}

\end{proposition}

\begin{remark}  Of course, these results are well-known. For $m > n$, the algebra of invariants
of $\rho$ can be identified with the homogeneous coordinate ring of the Grassmannian
$\operatorname{Gr}(n, m)$. 

\begin{itemize}

\item If $m = n+1$, then the Grassmannian $\operatorname{Gr}(n, m)$ is the projective space $\mathbb{P}^n$, whose homogeneous coordinate ring is the ring of polynomials in $n + 1$ variables, hence freely
generated. \item For $m > n+1$, the homogeneous coordinate ring of the Grassmannian $\operatorname{Gr}(n, m)$ 
is generated by Pl\"{u}cker coordinates, subject to Pl\"{u}cker relations. 

\end{itemize}

\end{remark}

\begin{proof}[Proof of Proposition~\ref{Prop:SLn}]

\begin{itemize}

\item There is a clearly no nonconstant invariants in the $m < n$ case, because in
this case the action of the group $\operatorname{SL}(n)$ on $m$-tuples of vectors in $\operatorname{R}^n$ has an open orbit. (One could also say that in this case there is no vertical indices and hence no invariants by
Theorem~\ref{T:SumDeg_Polyn}.) 

\item As for the case $m=n$, in this case we have one vertical index equal to $n$, and
a polynomial invariant $\operatorname{det}X$ of degree $n$. So, by Theorem~\ref{T:SumDeg_Polyn}, the algebra of invariants is freely generated by $\operatorname{det} X$ (which is obvious in this case, because the transcendence degree of the algebra of invariants is equal to 1, so it must be freely generated). 

\item Further, when $m = n+1$, the vertical indices are \[\underbrace{n, \dots ,n}_{n + 1 \text{ times}}\] which again coincides with the degrees of the $n \times n$ minors of $X$. So, these minors generate the algebra of invariants by Theorem~\ref{T:SumDeg_Polyn} (they are clearly independent, as one can find a matrix X with any prescribed values of these minors).

\item Finally, consider the case $m > n+1$. Observe that in this case there may also be no invariants of degree $k < n$. Indeed, assume that $f$ is such an invariant. Then there exist $k$ indices $1 \leq i_1 < \dots < i_k \leq n$ such that $f$ has non-trivial restriction to the subspace of matrices whose all columns, except for the ones with indices $i_1, \dots, i_k$, vanish. But this restriction must be an invariant of the representation $\operatorname{sl}(n) \to \operatorname{gl} \left( \operatorname{Mat}_{n \times k} \right)$, and hence trivial, which is a contradiction. So, any non-trivial invariant of $\rho$ has degree at least n. On the other hand, it is easy to see from Proposition~\ref{T:JKSumStandardSLn} that at least one of the vertical indices is less than n. So, by the second part of Theorem~\ref{T:SumDeg_Polyn}, the algebra of invariants is not freely generated, as stated (Theorem~\ref{T:SumDeg_Polyn} applies in our case, because $\operatorname{sl}(n)$  is simple and thus none of its representations admit proper semi-invariants).

\end{itemize}

Proposition~\ref{Prop:SLn} is proved. \end{proof}

\subsection{Standard representations of $\operatorname{so}(n)$ and $\operatorname{sp}(n)$ }

\subsubsection{JK invariants}

Let us now consider the Lie algebras $\operatorname{so}(n)$, which is the Lie algebra of skew-symmetic matrices, and $\operatorname{sp}(n)$. For $\operatorname{sp}(n)$ we always assume that $n$ is even $n=2k$ and we denote by $\operatorname{sp}(n)$ what would usually be denoted by $\operatorname{sp}(2k)$, that is the space of $n \times n$ matrices $X$ given by the equation \begin{equation} \label{Eq:SympLieAlg_Elem} X^T \Omega  + \Omega X = 0, \qquad \Omega ={\begin{pmatrix}0&I_{\frac{n}{2}}\\-I_{\frac{n}{2}}&0\\\end{pmatrix}}\end{equation} In other words, the elements of $\operatorname{sp}(n)$ have the form \begin{equation} \label{Eq:SympLieAlg_ElemOmegSym} X = \Omega^{-1} S, \qquad S^T =S.\end{equation} Note that for both $\operatorname{so}(n)$ and $\operatorname{sp}(n)$ the index $n$ is the dimension of the underlying space. 

\begin{remark} Since $\operatorname{so}(n)$ and $\operatorname{sp}(n)$ correspond to skew-symmetric and symmetric matrices of the same size, the obtained in the sequal results for them would usually differ by some choices of signs. For this reason it is convenient to write the formulas in terms of the number $\varepsilon$, where \begin{equation}\label{Eq:Epsilon} \varepsilon = \begin{cases} +1, \qquad \text{ for } \operatorname{sp}(n), \\ -1, \qquad \text{ for } \operatorname{so}(n). \end{cases} \end{equation}
\end{remark}

\begin{theorem} \label{T:JKSumStandardSOn} Let $\rho$ be the sum of m standard representations of $\operatorname{so}(n)$ or $\operatorname{sp}(n)$ and let $\varepsilon = \pm 1$ be given by \eqref{Eq:Epsilon}.

\begin{enumerate}

\item Assume that $m<n$. Divide $m$ by $n-m$ with the remainder: \[ m = q (n-m) + r, \qquad q, r \in \mathbb{Z}, \qquad 0 \leq r < n-m.\]  Then the JK invariants of $\rho$ are

 \begin{itemize}
 
 \item   $\displaystyle \frac{(n-m)(n-m +\varepsilon)}{2}$ horizontal indices: \[ \underbrace{2q+1, \dots, 2q+1}_{\frac{(n-m-r)(n-m-r + \varepsilon)}{2}}, \qquad \underbrace{2q+2, \dots, 2q+2}_{(n-m-r)r}, \qquad \underbrace{2q+3, \dots, 2q+3}_{ \frac{r(r + \varepsilon)}{2}},\] 

\item $\displaystyle \frac{m(m - \varepsilon)}{2}$ vertical indices $\mathrm{v}_i =2$.

\end{itemize}

\item  Assume that $m=n$. Then, in the $\operatorname{so}(n)$ case, the JK invariants of $\rho$ consist of $\displaystyle \frac{n(n+1)}{2} $ vertical indices:   \[ \underbrace{1, \dots, 1}_{n}, \qquad \underbrace{2, \dots, 2}_{\frac{n(n-1)}{2} },\]  while in the $\operatorname{sp}(n)$ case the JK invariants are  $\displaystyle \frac{n(n-1)}{2}$  vertical indices $\mathrm{v}_i =2$ and $n$  Jordan $1 \times 1$ blocks with different eigenvalues.

\item  Assume that $m>n$. Then the JK invariants of $\rho$ are $\displaystyle n(m-n) + \frac{n(n - \varepsilon)}{2}$ vertical indices: \[ \underbrace{1, \dots, 1}_{(m-n -\varepsilon)n}, \qquad \underbrace{2, \dots, 2}_{\frac{n(n + \varepsilon)}{2} }.\]

\end{enumerate}
\end{theorem}

\begin{remark} Similar to Remark~\ref{Rem:SumBlocksGL} we can check that the sum of sizes of blocks is \[\operatorname{dim} V \times \operatorname{dim} \mathfrak{g} = nm \times \frac{n(n+\varepsilon)}{2}.\] For $n < m$ the sum of ``widths'' of horizontal blocks is  \[ \frac{(n-m-r)(n-m-r + \varepsilon)}{2} \cdot 2q + (n-m-r)r \cdot (2q+1) + \frac{r(r + \varepsilon)}{2} \cdot (2q+2) = nm - m(m-\varepsilon).\] Hence the sum of ``heights'' of horizontal blocks is  \[  nm - m(m-\varepsilon) + \frac{(n-m)(n-m +\varepsilon)}{2} =  \frac{n(n+\varepsilon)}{2} -  \frac{m(m -\varepsilon)}{2}\] Adding the total size of vertical blocks we get the desired equality: \[ \left(nm - m(m-\varepsilon)\right) \times \left( \frac{n(n+\varepsilon)}{2} -  \frac{m(m -\varepsilon)}{2} \right) + \frac{m(m - \varepsilon)}{2}\left[ 2 \times 1\right] =  nm \times \frac{n(n+\varepsilon)}{2}. \] For $n=m$ and Lie algebra $\operatorname{so}(n)$ we have \[ n \left[1 \times 0 \right] + \frac{n(n-1)}{2}\left[2 \times 1\right] = n^2 \times \frac{n(n-1)}{2}.\] For $n=m$ and Lie algebra $\operatorname{sp}(n)$ we have \[ n \left[1 \times 1 \right] + \frac{n(n-1)}{2}\left[2 \times 1\right] = n^2 \times \frac{n(n+1)}{2}.\] For $n >m $, we have \[ n(m-n-\varepsilon) \left[1 \times 0 \right] + \frac{n(n+\varepsilon)}{2}\left[2 \times 1\right] = nm \times \frac{n(n+\varepsilon)}{2}.\] All the necessary equalities hold.
\end{remark}

\subsubsection{Number and types of blocks}

From Propositions~\ref{Prop:EigenSing}, \ref{Prop:JordBlock_EigenNumber} and \ref{nhnv} we get the following simple statement.

\begin{proposition} \label{Prop:SO_SP_SingStOreg} For the sum of $m$ standard representations of $\operatorname{so}(n)$ and $\operatorname{sp}(n)$ we have the following:

\begin{enumerate}

\item The singular set is \[ \Sing = \left\{ \left. A \in \operatorname{Mat}_{n \times m} \right| \rk A \leq \min{(n, m)}  \right\} \]   except for the case $n = m$ for $\operatorname{so}(n)$, when \[ \Sing = \left\{ \left. A \in \operatorname{Mat}_{n \times m} \right| \rk A \leq n-2  \right\}. \] 

\item There are Jordan blocks in the JK invariants only for $m=n$ for $\operatorname{sp}(n)$. In that case there are $m=n$ Jordan $1\times 1$ blocks with distinct eigenvalues $\lambda_i$, since $\Sing$ is defined by one polynomial $\det A$ and \[\dim \operatorname{St}_{X-\lambda_i A} - \dimSt =1\]  for the generic pair $(X, A)$ from Lemma~\ref{L:GenMatPairSL}.

\item The numbers of horizontal and vertical indices are \[ \dimSt = \begin{cases} \frac{(n-m)(n-m + \varepsilon)}{2},& \quad m <n, \\ 0,& \quad m \geq n,\end{cases} \qquad \codimO = \begin{cases} \frac{m(m- \varepsilon)}{2}, &\quad m\leq n, \\ n(m-n) + \frac{n(n - \varepsilon)}{2}, & \quad m > n,\end{cases} \] respectively. 

\end{enumerate}

Here $\varepsilon$ is given by \eqref{Eq:Epsilon}. \end{proposition}

Then we can find the number of indices equal to $1$ using Proposition~\ref{Prop:Triv_KroneckerBlocks_TotalNumber}.

\begin{proposition} 
\label{Prop:SO_SP_SumStandRepr_TrivKronBlocks}
Consider a generic pencil $R_{X }+\lambda R_A$ for the sum of $m$ standard representations of $\operatorname{so}(n)$ or $\operatorname{sp}(n)$. The number of horizontal indices equal to $1$ is
\[ \dim \left(\operatorname{St}_X \cap \operatorname{St}_A\right) =\left\{ \begin{array}{lc} \displaystyle \frac{(n-2m)(n-2m + \varepsilon)}{2}, & \qquad 2m< n \\ 0, &  \qquad  2m \geq  n\end{array} \right. \] and the number of vertical indices equal to $1$ is  \[  \codim \left(\operatorname{Im} R_X + \operatorname{Im} R_A \right) = \left\{ \begin{array}{lc}  0, &  \qquad  m <n \\ 0, & \quad m=n, \quad \operatorname{sp}(n) \\ n, & \quad m=n,  \quad \operatorname{so}(n)  \\   (m-n - \varepsilon)n, & \quad m > n \end{array} \right. \] Here $\varepsilon$ is given by \eqref{Eq:Epsilon}. \end{proposition}

In order to prove Proposition~\ref{Prop:SO_SP_SumStandRepr_TrivKronBlocks} we will need the following simple statement about matrices.

\begin{proposition} \label{Prop:MatTwoMinSym} Consider the space of $n\times m$ matrices, where $m>n$. 

\begin{itemize} \item Then the dimension of its subspace such that the $n\times n$ submatrices formed by the first $n$ and the last $n$ columns are both symmetric is $(m-n+1) n$. \item And if the $n\times n$ submatrices are skew-symmetric, then the dimension of the subspace is $(m-n-1) n.$  \end{itemize} \end{proposition}

In Proposition~\ref{Eq:Epsilon} we discuss matrices of the following kind:
 
 \begin{equation} \label{Eq:MatSymSkewSubMat_Exam} \begin{pmatrix} \cellcolor{blue!15} a_1 & \cellcolor{blue!15}  a_2 & \cellcolor{blue!15}  a_3 & a_5 & a_{6} & a_7 \\ 
 a_2 & \cellcolor{blue!15}  a_4 &  \cellcolor{blue!15}  a_5 & \cellcolor{blue!15} a_6 & a_{7} & a_{8} \\
 a_3 & a_5 & \cellcolor{blue!15} a_6 & \cellcolor{blue!15} a_7 & \cellcolor{blue!15} a_{9} & a_{10} \\
 a_5 & a_6 & a_{7} & \cellcolor{blue!15} a_{8} & \cellcolor{blue!15} a_{10} & \cellcolor{blue!15} a_{11} \\
\end{pmatrix}, \qquad  \left( \begin{matrix} 0 & \cellcolor{blue!15} a_2 & 0 & a_5 & 0 & a_7 \\ 
 -a_2 & 0 & \cellcolor{blue!15} -a_5 & 0 & -a_{7} & 0 \\
0 & a_5 & 0 & \cellcolor{blue!15} a_7 & 0 & a_{10} \\
 -a_5 & 0 & -a_{7} & 0 & \cellcolor{blue!15} -a_{10} & 0 \\
\end{matrix} \right) \end{equation}

\begin{proof}[Proof of Proposition~\ref{Prop:MatTwoMinSym}] It is obvious that such matrices are completely determined by the elements in the ``strip'' between the main diagonals of the $n\times n$ matrices. (This ``strip'' is colored in \eqref{Eq:MatSymSkewSubMat_Exam} for better visualization.) Here the elements on the main diagonals are included for symmetric submatrices and must be equal to zero for skew-symmetric submatrices. Proposition~\ref{Prop:MatTwoMinSym} is proved. \end{proof}

 \begin{proof}[Proof of Proposition~\ref{Prop:SO_SP_SumStandRepr_TrivKronBlocks}]  By Lemma~\ref{L:GenMatPairSL}  we can assume that the pair $(X, A)$ is as in Lemma~\ref{L:GenMatPair}.  The statement about horizontal indices is proved by direct computation. For vertical indices it is more convenient to consider dual operators and find $\dim \left(\operatorname{Ker} R_X^* \cap \operatorname{Ker} R_A^* \right)$. Let us identify $\operatorname{Mat}_{n \times m}^*$ with $\operatorname{Mat}_{m \times n}$   by the formula \begin{equation} \label{Eq:TrMatProd} \langle A, B \rangle = \operatorname{tr} \left(A B \right).  \end{equation} Notice that we study a restriction of a linear map $\gl(n) \to \operatorname{Mat}_{n \times m}$ to a matrix subalgebra $\mathfrak{h}= \operatorname{so}(n)$ or $\operatorname{sp}(n)$. Thus the dual map $\operatorname{Mat}^*_{n \times m} \to \mathfrak{h}^*$ is the composition of the dual map $ \operatorname{Mat}^*_{n \times m} \to \gl(n)^* \approx \operatorname{Mat}^*_{n \times n}$ and the natural projection of $\operatorname{Mat}^*_{n \times n}$ to $\mathfrak{h}^*$.

\begin{assertion} 
\label{A:MatrixDualTransp} 
Identify $\operatorname{Mat}_{n \times m}^*$ with $\operatorname{Mat}_{m \times n}$  by the formula \eqref{Eq:TrMatProd}. 
\begin{enumerate}

\item The dual to the right multiplication $R_X: \operatorname{Mat}_{n \times n} \to \operatorname{Mat}_{n \times m}, R_X(Y)= YX$ is the left multiplication $L_X: \operatorname{Mat}_{m \times n} \to \operatorname{Mat}_{n \times n}, L_X( Y)= XY$.

\item Symmetric and skew-symmetric matrices from $\operatorname{Mat}_{n \times n}$ are orthogonal w.r.t. the inner product given by \eqref{Eq:TrMatProd}.

\end{enumerate}

\end{assertion}

It follows that $\operatorname{Ann}(\operatorname{so}(n))$ in $\operatorname{gl}(n)^* \approx \operatorname{Mat}_{n \times n}$ consits of symmetric matrices. We assume that $m>n$, the case $m\leq n$ is considered analogously. Since $R_X^*$ and $R_A^*$ are left multiplications, $\operatorname{Ker} R_X^* \cap \operatorname{Ker} R_A^*$ for $\operatorname{so}(n)$ consists of $m\times n$ matrices such that their  $n\times n$ submatrices formed by the first $n$ and the last $n$ rows are both symmetric. The dimension of such space is found in Proposition~\ref{Prop:MatTwoMinSym}. 

We have proved the statement for $\operatorname{so}(n)$. For $\operatorname{sp}(n)$ by Assertion~\ref{A:MatrixDualTransp}  $\operatorname{Ann}(\operatorname{sp}(n))$ consists of matrices $\Lambda \Omega$, where $\Lambda^T = - \Lambda$ and the proof is similar. Proposition~\ref{Prop:SO_SP_SumStandRepr_TrivKronBlocks} is proved. \end{proof}

\subsubsection{Proof of Theorem~\ref{T:JKSumStandardSOn}}

So far, we saw two ways to calculate the JK invariants of a representation:

\begin{enumerate}

\item One can find a generic pair $(X, A)$ and then describe the bases of all blocks in the JK decomposition, similar to the proof of Theorem~\ref{T:JKSumStandardSLn}.

\item One can estimate the numbers and types of blocks, e.g. we did it for $\operatorname{sl}(n)$ in Section~\ref{SubS:BlockTypesSln}, using propositions from Section~\ref{S:InterpretJK}. 

\end{enumerate}

The first way is usually quite tiresome. For the sums of standard representations of $\operatorname{so}(n)$ and $\operatorname{sp}(n)$ it is easier to estimate the number of blocks using results from Section~\ref{S:InterpretJK}.

\begin{proof}[Proof of Theorem~\ref{T:JKSumStandardSOn}]  For $m \geq n$ the number and types of blocks follow from Propositions~\ref{Prop:SO_SP_SingStOreg} and \ref{Prop:SO_SP_SumStandRepr_TrivKronBlocks} and considerations of dimensions. 

For the rest of the proof assume that $m < n$. From Propositions~\ref{Prop:SO_SP_SingStOreg} and \ref{Prop:SO_SP_SumStandRepr_TrivKronBlocks}  we know that there are no Jordan blocks and that there are $\displaystyle \frac{m(m-\varepsilon)}{2}$ vertical indices $\mathrm{v}_i \geq 2$. Thus, it suffices to prove the statement about the horizontal indices, since vertical indices take minimal possible values. 

We find horizontal indices $\mathrm{h}_i$ using Proposition~\ref{L:NumberHorizontInd}. 
First, we take a generic pair $(X, A)$ from Lemma~\ref{L:GenMatPair} (it is generic by Lemma~\ref{L:GenMatPairSL}). To simplify indices, it is more convenient to exchange $X$ and $A$ (this doesn't change horizontal indices). Thus, we search for $\displaystyle \frac{(n-m)(n-m+\varepsilon)}{2}$ sequences $\Lambda_0, \dots, \Lambda_\alpha \in \mathfrak{g}$ (where $\mathfrak{g} =  \operatorname{so}(n)$ or $ \operatorname{sp}(n)$)  such that 
\begin{equation} \label{Eq:KroneckerBlocks_ImageBasis_XASwich} \Lambda_0 A=0, \qquad \Lambda_1 A = \Lambda_0 X, \qquad  \dots \qquad \Lambda_{\alpha} A = \Lambda_{\alpha-1} X, \qquad \Lambda_{\alpha} X=0 \end{equation} and the vectors $\Lambda_j$ for all these sequences are linearly independant.

\begin{itemize}

\item Case $\operatorname{so}(n)$.  Let us describe the  sequences $\Lambda_0, \dots, \Lambda_\alpha \in \operatorname{so}(n)$. Just as in the proofs of Theorem~\ref{T:JKSumStandardSLn} for $m<n$ if for some matrix any of its indices goes beyond $[1, n]$, then we put this matrix down to zero. For each $(i,j)$ such that \[ 1 \leq i < j  \leq n-m \] we take the sequence \begin{equation} \label{Eq:SO_HorBlocksBasis} \Lambda_k = \sum_{l=0}^k  \left( E_{i + l(n-m), j + (k-l)(n-m) }  - E_{ j + (l-k)(n-m), i + k(n-m)}  \right), \end{equation}  where \begin{equation} \label{Eq:SO_HorChainIndexCond}  i \leq i + l(n-m) \leq n , \qquad j \leq j + (k-l)(n-m)  \leq n  \end{equation}  Note that all non-zero elements of \eqref{Eq:SO_HorBlocksBasis} lie on one skew-diagonal, since the sum of indices is always $i+j +k(n-m)$. The matrix of \eqref{Eq:SO_HorBlocksBasis} has the form \[ \left( \begin{matrix} \Omega_0 & \Omega_1 & \Omega_2 & \dots \\ \Omega_1 & \Omega_2 & \Omega_3 &  \dots \\ \Omega_2 & \Omega_3 & \Omega_4 & \dots \\ \vdots & \vdots & \vdots & \ddots \end{matrix}  \right), \] where $\Omega_k = E_{ij} - E_{ji}$ and all other $\Omega_{r}=0$. If only a part of a submatrix $\Omega_k$ ``fits'' into the matrix we take only that ``fitting'' submatrix of $\Omega_k$.

All matrices $\Lambda_k$ are skew-symmetric and hence belong to $\operatorname{so}(n)$.  The condition \eqref{Eq:KroneckerBlocks_ImageBasis_XASwich} is satisfied due to \eqref{Eq:ChainEq_MLessN}. The ``terminal'' conditions $\Lambda_0 A=0$ and $\Lambda_{\alpha} X =0$ are satisfied since \[ E_{i, j} X= \begin{cases} E_{i, j}, \quad & 1 \leq j \leq m \\ 0, \qquad & \text{otherwise}. \end{cases} \qquad   E_{i, j} A = \begin{cases} E_{i, j-(n-m)}, \quad & 1+(n-m) \leq i \leq n \\ 0, \qquad & \text{otherwise}. \end{cases} \]

The horizontal index corresponding to each block with basis \eqref{Eq:SO_HorBlocksBasis} is the maximal possible $k$ such that there exists $l$ satisfying \eqref{Eq:SO_HorChainIndexCond} plus $1$. There are $3$ cases:
\begin{enumerate}

\item If $i \in [1, q],  j\in [1, q]$, then $\mathrm{h}_i=2q+3$.

\item If $i \in [1, q],  j  \not \in [1, q]$, then $\mathrm{h}_i=2q+2$.

\item If $i \not \in [1, q],  j  \not \in [1, q]$, then $\mathrm{h}_i=2q+1$.

\end{enumerate}

It can be easily checked that the total number of horizontal indices of each type is $\frac{r(r-1)}{2}$, $(n-m-r)r$ and $\frac{(n-m-r)(n-m-r-1)}{2}$  respectively.

\item Case $\operatorname{sp}(n)$. The proof is similar to the previous case $\operatorname{so}(n)$. Only when constructing the solutions of \eqref{Eq:KroneckerBlocks_ImageBasis_XASwich}, we replace sequences \eqref{Eq:SO_HorBlocksBasis} by \begin{equation} \label{Eq:SO_HorBlocksBasis} \Lambda_k = \Omega^{-1} \sum_{l=0}^k  \left( E_{i + l(n-m), j + (k-l)(n-m) }  + E_{ j + (l-k)(n-m), i + k(n-m)}  \right), \end{equation}  where \begin{equation} \label{Eq:SO_HorChainIndexCond}  i \leq i + l(n-m) \leq n , \qquad j \leq j + (k-l)(n-m)  \leq n  \end{equation}  Their matrices have the form \[ \Omega^{-1} \left( \begin{matrix} S_0 & S_1 & S_2 & \dots \\ S_1 & S_2 & S_3 &  \dots \\ S_2 & S_3 & S_4 & \dots \\ \vdots & \vdots & \vdots & \ddots \end{matrix}  \right), \] where $\Omega$ is given by \eqref{Eq:SympLieAlg_Elem}, $S_k = E_{ij} + E_{ji}$ and all other $S_{r}=0$.

Here the matrices $\Lambda_k$ belong to $\operatorname{sp}(n)$ since they have the form \eqref{Eq:SympLieAlg_ElemOmegSym}. Comparing to the case of $\operatorname{so}(n)$ there are $r$ new indices $\mathrm{h}_i = 2q+1$, corresponding to $1 \leq i=j \leq r$ and $n-m-r$ new indices $\mathrm{h}_i = 2q+3$, corresponding to $r< i=j \leq n-m $. 
\end{itemize}

Theorem~\ref{T:JKSumStandardSOn} is proved.

\end{proof}

\subsubsection{Algebra of invariant polynomials}

The following proposition and its proof are taken from \cite{BolsIzosKozl19} for the sake of completeness.

\begin{proposition} \label{Prop:SOnInvar} Let $\rho : \operatorname{so}(n) \to \operatorname{gl} \left( \operatorname{Mat}_{n \times m}\right)$ be the sum of $m$ standard representations of $\operatorname{so}(n) $, and let $X \in \operatorname{Mat}_{n \times m}$. Then:
\begin{enumerate}

\item For $m < n$, the algebra of invariants of $\rho$ is freely generated by $\frac{m(m+1)}{2}$ pairwise inner
products of columns of $X$.

\item For $m \geq n$, the algebra of invariants of $\rho$ is not freely generated.
\end{enumerate}

\end{proposition}

\begin{proof}[Proof of Proposition~\ref{Prop:SOnInvar}]

\begin{enumerate}

\item  The first statement is a direct consequence of Theorem~\ref{thm3}, since all pairwise inner products of columns of X have degree 2 and hence coincide with
vertical indices. 

\item To prove the second statement, observe that the same restriction argument
as we used in the proof of Proposition~\ref{Prop:SLn} shows that there can be no invariants of
degree 1. On the other hand, some of the vertical indices are equal to 1, so the algebra of
invariants is not freely generated by Theorem~\ref{thm3}.
\end{enumerate}

Proposition~\ref{Prop:SOnInvar} is proved. \end{proof}

\begin{remark} For $m \geq n$, the algebra of invariants of $\rho$ is known to be generated by
pairwise inner products of columns of $X$, i.e. \[ g_{ij} = \left( X_i, X_j \right) = \sum_{k=1}^n X_{ik} X_{jk}, \] supplemented by $n \times n$ minors of $X$. It is easy
to see that pairwise inner products alone do not generate the algebra of invariants, as any
minor of $X$ can be expressed as the square root of the Gram determinant of its columns,
which is a non-polynomial function in terms of the pairwise inner products. \end{remark}

The next statement is proved similarly to Proposition~\ref{Prop:SOnInvar}.

\begin{proposition} \label{Prop:SPnInvar} Let $\rho : \operatorname{sp}(n) \to \operatorname{gl} \left( \operatorname{Mat}_{n \times m}\right)$ be the sum of $m$ standard representations of $\operatorname{sp}(n) $, and let $X \in \operatorname{Mat}_{n \times m}$. Then:
\begin{enumerate}

\item For $m \leq n + 1$, the algebra of invariants of $\rho$ is freely generated by $\frac{m(m-1)}{2}$ pairwise
symplectic products of columns of $X$.

\item For $m > n  + 1$, the algebra of invariants of $\rho$ is not freely generated.
\end{enumerate}

\end{proposition}

\begin{remark} In contrast to the $\operatorname{so}(n)$ case, the minors of $X$ in the symplectic case can
be expressed in terms of pairwise symplectic products  \[\Omega_{ij}  = \Omega\left( X_i, X_j \right) = \sum_{k,l=1}^n \Omega_{kl} X_{ik} X_{jl}. \] Namely, any minor is equal to the
Pfaffian of the symplectic Gram matrix of its columns. So, the algebra of invariants of
$\rho$ is generated just by the symplectic products. The relations between these symplectic
products for $m > n + 1$ are given by Pl\"{u}cker  relations between minors of $X$. \end{remark}

\subsection{Standard representations of $\operatorname{b}(n)$ }

\begin{theorem} \label{T:JKSumStandardBn} Let $\rho$ be the sum of m standard representations of the Lie algebra of upper triangular matrices $\operatorname{b}(n)$.

\begin{enumerate}

\item Assume that $m \leq n$. Then the JK invariants of $\rho$ are

 \begin{itemize}
 
 \item $\frac{(n-m)(n-m+1)}{2}$ horizontal indices. Namely, for each $1 \leq j \leq n-m$ we take $j$ following horizontal indices. Divide $m$ by $j$ with the remainder \[ m= q_j \cdot j + r_j, \qquad  0 \leq r_j < j.\] The corresponding horizontal indices are \[ \underbrace{q_j+1, \dots, q_j+1}_{j-r_j}, \qquad \underbrace{q_j+2, \dots, q_j+2}_{r_j}.\] We take the union of all these indices for $1 \leq j \leq n-m$.
 
 \item $m$ Jordan $1\times 1$ blocks with distinct eigenvalues.
 
 \item $\frac{m(m-1)}{2}$ vertical indices.  Namely, for each $1 \leq j \leq m-1$  we take $j$ following vertical indices. Divide $m$ by $j$ with the remainder \[ m= q_j \cdot  j + r_j, \qquad  0 \leq r_j < j.\] The corresponding vertical indices are \[ \underbrace{q_j, \dots, q_j}_{j-r_j}, \qquad \underbrace{q_j+1, \dots, q_j+1}_{r_j}.\] We take the union of all these indices for $1 \leq j \leq m-1$.
 
 \end{itemize}
 
Note that there are no horizontal indices for $m = n$. 

\item  Assume that $m>n$. Then the JK invariants of $\rho$ are $nm - \frac{n(n+1)}{2}$ vertical indices. Namely, for each $1 \leq j \leq n$  we take $m-j$ following vertical indices. Divide $m$ by $m-j$ with the remainder  \[ m= q_j (m-j) + r_j, \qquad  0 \leq r_j <m-j.\] The corresponding vertical indices are \[ \underbrace{q_j, \dots, q_j}_{m-j - r_j}, \qquad \underbrace{q_j+1, \dots, q_j+1}_{r_j}.\] We take the union of all these indices for $1 \leq j \leq n$.

\end{enumerate}
\end{theorem}

\begin{remark} Similar to Remark~\ref{Rem:SumBlocksGL} we can check that the sum of sizes of blocks is \[\operatorname{dim} V \times \operatorname{dim} \mathfrak{g} = nm \times \frac{n(n+1)}{2}.\] 

\begin{itemize}

\item  For $n \leq m$ the sum of ``widths'' of horizontal blocks is  \[ \sum_{j=1}^{n-m} (q_j+1)(j-r_j) + (q_j + 2) r_j = \sum_{j=1}^{n-m}  m + j = \frac{(n-m)(n+m+1)}{2}.\] Since there are $\frac{(n-m)(n-m+1)}{2}$ such blocks their total sum of sizes is \[  (n-m)m \times  \frac{(n-m)(n+m+1)}{2}.\]  The sum of ``heights'' of vertical blocks is  \[ \sum_{j=1}^{m-1} q_j (j-r_j) + (q_j+1) r_j =  \sum_{j=1}^{m-1} m = m(m-1).\] Hence the total size of vertical blocks is  \[ m(m-1) \times \frac{m(m-1)}{2}.\] In total, since the Jordan blocks add $m\times m$ to the total size we get \[ (n-m)m \times  \frac{(n-m)(n+m+1)}{2} + m \times m +m(m-1) \times \frac{m(m-1)}{2}  = nm \times \frac{n(n+1)}{2}.\] 

\item For $m >n $, the sum of ``heights'' of blocks is \[ \sum_{j=1}^n (m-j - r_j) q_j  + r_j (q_j + 1) = \sum_{j=1}^n m = nm.\] Thus the total sum of sizes is \[ nm \times \left(nm - nm + \frac{n(n+1)}{2}  \right)= nm \times \frac{n(n+1)}{2}.  \] 
\end{itemize}

All the necessary equalities hold.

\end{remark}

\subsubsection{Number and types of blocks}

Let $\operatorname{B}(n)$ be the group of invertible upper triangular matrices (i.e. the Borel subgroup of $\operatorname{GL}(n)$) and let $\operatorname{b}(n)$ be its Lie algebra, consisting of upper triangular matrices. We want to simplify the matrices of a generic pair $X, A$ for the sum of standard representations of $\operatorname{b}(n)$. For that we prove the following statement, which is similar to Proposition~\ref{Prop:SL_SumRepr_LeftRightAction}.

\begin{proposition}  \label{Prop:B_SumRepr_LeftRightAction} Consider the sum of $m$ standard representations for the Lie algebra $\operatorname{b}(n)$.  Then for any invertible upper triangular matrix $C \in \operatorname{B}(n)$ and any non-degenerate matrix $D \in \operatorname{GL}(m)$ the JK decompositions of a pencil $R_X +\lambda R_A$ and the pencil $R_{CXD} + \lambda R_{CAD}$ coincide.\end{proposition}

\begin{proof}[Proof of Proposition~\ref{Prop:B_SumRepr_LeftRightAction}] 
It suffices to prove the existence of invertible linear mappings $\varphi: \operatorname{b}(n) \to\operatorname{b}(n)$ and $\psi: \operatorname{Mat}_{n \times m} \to \operatorname{Mat}_{n \times m}$ (where $\varphi$ does not have to be a Lie algebra automorphism) such that the following diagram commutes\begin{equation} \label{Eq:BnMatCommDiag}
\begin{CD}
\operatorname{b}(n)  @>R_{CXD} + \lambda R_{CAD} >> \operatorname{Mat}_{n \times m} \\
@VV \varphi V @AA\psi A\\
\operatorname{b}(n) @>R_X + \lambda R_A>> \operatorname{Mat}_{n \times m}
\end{CD}
\end{equation} For  $\operatorname{b}(n)$ a suitable choice of $\varphi$, $\psi$ is \[\varphi(Y) = C^{-1}YC, \qquad \psi(Z) = C ZD. \] The automorphisms $\varphi$, $\psi$  intertwine the pencils  $R_{X + \lambda A}$ and  $R_{CXD + \lambda CAD}$, so those pencils have the same JK type, as desired. Proposition~\ref{Prop:SL_SumRepr_LeftRightAction} is proved. \end{proof}

Now, let us study, how we can simplify an  $n \times m$ matrix by the left-right action of $\operatorname{B}(n) \times \operatorname{GL}(m)$.

\begin{proposition} \label{Prop:LeftRightBGL} For any matrix $A' \in \operatorname{Mat}_{n\times  m}$ there exists $C \in \operatorname{B}(n)$ and $D \in \operatorname{GL}(m)$ such that $A = C A' D$ has the form \begin{equation} \label{Eq:UpperGenMatmn} A =E_{i_1, 1} + \dots E_{i_k, k}, \qquad k \leq \min{(m, n)}, \qquad 1 \leq  i_1 <i_2 < \dots <i_k \leq n \end{equation}
\end{proposition}

Note that $\rk A = k$. For example, consider matrices \eqref{Eq:UpperGenMatmn} with maximal rank $k = \min(m, n)$. 

\begin{itemize}

\item If $m \geq n$, then $A = \left(\begin{matrix} I_n & 0 \end{matrix}  \right)$. 

\item And if $m < n$, then $A$ has the form \[A= \left( \begin{matrix} 
0 & \dots & \dots & 0 \\ 
\vdots & \vdots & \vdots & \vdots \\ 
1 & \vdots & \vdots & \vdots  \\ 
 & \vdots & \vdots & \vdots \\ 
 & 1 & \vdots & \vdots \\  
 &  & \vdots & \vdots \\ 
 & & \vdots & \vdots \\
&  & & \vdots \\ 
&  & & 0 \\ 
 &  &  & 1 \\ 
 &  &  &  0\\ 
 &  & & \vdots \\ 
\end{matrix} \right). \]

\end{itemize}

\begin{proof}[Proof of Proposition~\ref{Prop:LeftRightBGL}] Right multiplying  $A$ on a non-degenerate matrix $D \in \operatorname{GL}(m)$ we can achieve the following:
\begin{enumerate}

\item first $k=\rk A$ columns $v_1, \dots, v_k$ of $A$ are non-zero, last $n-k$ columns are zero columns.

\item let $i_j$ be the maximal index such that $v_{j,i_j} \not = 0$, i.e. $v_{j, s} = 0$ for $s > k_j$. Then \[ 1 \leq i_1 < i_2 < \dots < i_k   \leq n.\]

\end{enumerate} Roughly speaking, $A$ takes the following form:
\begin{equation} \label{Eq:AColEcheBn} A= \left( \begin{matrix} X & 0 \end{matrix}  \right), \qquad X = \left( \begin{matrix} 
* & * & \vdots & *  \\ 
1 & * & \vdots & * \\ 
 & * & \vdots & * \\ 
 & 1 & \vdots & * \\  
 &  & \vdots & * \\ 
 & & \vdots & * \\
 &  &  & 1 \\ 
 &  &  &  \\ 
 &  & &  \\ 
\end{matrix} \right). \end{equation}
Denote by $A^*$ the matrix that is symmetic to $A$ about its center, that is \[A^*_{i, j} = A_{n-i+1, m-j+1}.\] Note that if $A$ has the form \eqref{Eq:AColEcheBn}, then $A^*$ is in reduced column echelon form. Hence, reduction of $A$ to the form  \eqref{Eq:AColEcheBn} via right multiplications on $D \in \operatorname{GL}(n)$ can be viewed as the Gauss elimination (more precisely, column reduction) for $A^*$. 

Next, for $A$ given by  \eqref{Eq:AColEcheBn} we can easily get rid of the non-null values ``above ones'' by left multiplication on an upper triangular matrix $C \in \operatorname{B}(n)$. We just need to take an invertible upper triangular matrix $C$ such that \[C v_i = e_{i_j}, \quad i = 1, \dots, k.\]  After these left-right multiplications $A$ takes the required form \eqref{Eq:UpperGenMatmn}. Proposition~\ref{Prop:LeftRightBGL} is proved. \end{proof}

Next, we determine which elements \eqref{Eq:UpperGenMatmn} are regular.

\begin{proposition} \label{Prob:regBn} Let $A \in \operatorname{Mat}_{n \times m}$ be equivalent under the left-right $\operatorname{B}(n) \times \operatorname{GL}(m)$ action to \eqref{Eq:UpperGenMatmn}. Then we have \[ \dim \St_{A} = \frac{n(n+1)}{2} - \left(i_1 + \dots + i_k\right). \] 
\end{proposition}

\begin{proof}[Proof of Proposition~\ref{Prob:regBn}] Recall that $\operatorname{St}_A = \Ker R_A$. Since the diagram \eqref{Eq:BnMatCommDiag} commutes, $\dim \operatorname{St}_{CAD} = \dim \operatorname{St}_A$ for any $C \in \operatorname{B}(n)$ and $D \in \operatorname{GL}(m)$. Thus, we can assume that $A$ has the form \eqref{Eq:UpperGenMatmn}. Then elements of $\St_{A}$ are upper triangular matrices with zeroes in $i_j$-th columns ($j=1, \dots k$). Proposition~\ref{Prob:regBn} is proved. \end{proof} 

It is easy to see what elements  \eqref{Eq:UpperGenMatmn} are regular. 

\begin{corollary} \label{Cor:BnReg} Consider the sum of m standard representations of $\operatorname{b}(n)$.  The regular elements are $X \in \operatorname{Mat}_{n \times m}$ given by the following conditions:
\begin{itemize}

\item If $m \leq n$, then $ M_m \not = 0 $, where $M_m$ is the down-right minor of $X$ of order $m$;

\item If $m > n$, then $\rk X = n$. 

\end{itemize} For any regular $X$ there exist $C \in \operatorname{B}(n)$ and $D \in \operatorname{GL}(m)$ such that $CXD$ is $\left( \begin{matrix} 0 \\ I_m \end{matrix} \right)$ for $m \leq n$  or $\left( \begin{matrix} I_n & 0 \end{matrix} \right)$ for $m > n$. 
 \end{corollary}

\begin{proof}[Proof of Corollary~\ref{Cor:BnReg}]  By Proposition~\ref{Prob:regBn} the  elements $\left( \begin{matrix} 0 \\ I_m \end{matrix} \right)$ and $\left( \begin{matrix} I_n & 0 \end{matrix} \right)$ are the only regular elements of the from \eqref{Eq:UpperGenMatmn}. The left-right action 
$\operatorname{B}(n) \times \operatorname{GL}(m)$ preserves $\rk X$ and if $m < n$, it also preserves the condition $M_n \not = 0$. From the proof of Proposition~\ref{Prop:LeftRightBGL} it is easy to see that $X$ with $\rk X = n$ for $m \geq n$ or with $M_m \not = 0$ for $m < n$ are equivalent under the left-right action to $\left( \begin{matrix} 0 \\ I_m \end{matrix} \right)$ and $\left( \begin{matrix} I_n & 0 \end{matrix} \right)$ and thus they are regular. Corollary~\ref{Cor:BnReg} is proved. \end{proof}

Now we can easily calculate stabilizers for elements from Corollary~\ref{Cor:BnReg} and get the following. 

\begin{corollary} \label{Cor:Bn_SingStOreg} For the sum $\rho$ of $m$ standard representations of $\operatorname{b}(n)$ we have the following.

\begin{enumerate}

\item The singular set is as follows: 

\begin{itemize}

\item for $m < n$ we have \[ \Sing = \left\{ \left. A \in \operatorname{Mat}_{n \times m} \right| \quad M_m = 0\right\}, \]   where $M_m$ is the down-right minor of order $m$, i.e. determinant of the submatrix obtained by deleting first $n-m$ columns and rows;

\item for $m \geq n$ we have \[ \Sing = \left\{ \left. A \in \operatorname{Mat}_{n \times m} \right| \rk A < n  \right\}. \]  

\end{itemize}

\item Thus, JK invariants of $\rho$ contain Jordan blocks only for $m \leq n$ and there are Jordan blocks with $m$ distinct eigenvalues.  

\item The numbers of horizontal and vertical indices are \[ \dimSt=\begin{cases} \displaystyle \frac{(n-m)(n-m+1)}{2}, & \quad m<n \\ 0, & \quad m \geq n. \end{cases}, \qquad \codimO=\begin{cases} \displaystyle \frac{m(m-1)}{2}, &\quad m\leq n \\ \displaystyle nm - \frac{n(n+1)}{2}, &\quad m > n. \end{cases} \] 

\end{enumerate}

\end{corollary}

\subsubsection{Proof of Theorem~\ref{T:JKSumStandardBn} }

It suffices to find the Kronecker indices, since we know the number of eigenvalues from Corollary~\ref{Cor:Bn_SingStOreg}. The Jordan blocks take minimal possible values: there can't be less than one $1\times 1$ block for each eigenvalue. We calculate Kronecker indices in several steps, using Proposition~\ref{L:NumberHorizontInd}.

\paragraph{Calculation of horizontal indices.} By Corollary~\ref{Cor:Bn_SingStOreg} $m < n$ and by Proposition~\ref{Prop:B_SumRepr_LeftRightAction} and Corollary~\ref{Cor:BnReg} we can assume that $X=\left( \begin{matrix} 0 \\ I_m \end{matrix} \right)$. In order to apply Proposition~\ref{L:NumberHorizontInd} we seach for upper triangular $U_0, \dots, U_j \in \operatorname{b}(n)$  that satisfy \begin{equation} \label{Eq:horBn}  \left( U_0 + \lambda U_1 + \dots \lambda^j U_j \right)  \left(X - \lambda A \right) =0. \end{equation} Transpose \eqref{Eq:horBn}: \begin{equation} \label{Eq:horBnTrans}  \left(X^T - \lambda A^T \right) \left( L_0 + \lambda L_1 + \dots \lambda^j L_j \right)  =0, \end{equation}  where $L_i = U_i^T$ are lower triangular matrices. Note that we can solve \eqref{Eq:horBnTrans} column-wise. Then solutions of \eqref{Eq:horBnTrans} correspond to horizontal indices of $X^T$ and $A^T$, let us calculate them.

\begin{proposition} \label{Prop:GenXASpecificX} Let $m < n$, and $X =\left( \begin{matrix} 0 & I_m \end{matrix} \right)$ be a $m \times n$ matrix. If \[ m = q(n-m) +r, \qquad 0 \leq r < n-m,\]  then  for a generic $A \in \operatorname{Mat}_{m \times n}$  the JK invariants of a pair $X, A$ are $(n-m)$ horizontal indices: \[  \underbrace{q + 1, \dots, q + 1}_{n-m-r}, \qquad \underbrace{q+2, \dots, q+2}_{r}.\]
\end{proposition}

\begin{proof}[Proof of Proposition~\ref{Prop:GenXASpecificX}] For a generic pair $X', A' \in \operatorname{Mat}_{m \times n}$ there exist $C \in \operatorname{GL}(m)$ and $D \in \operatorname{GL}(n)$ such that $CX'D = \left( \begin{matrix} 0 & I_m \end{matrix} \right)$. By Proposition~\ref{Prop:GL_SumRepr_LeftRightAction} the JK invariants are invariant under the left-right $\operatorname{GL}(m) \times \operatorname{GL}(n)$  action, and thus the JK invariants for $X, A$ are as in Corollary~\ref{Cor:GenPenJK} for a generic pair $X', A'$. Proposition~\ref{Prop:GenXASpecificX} is proved.
\end{proof}

Consider \eqref{Eq:horBnTrans}  for the $i$-th column. Since $L_j$ are lower triangular, the $i$-th column consist of vectors with first $i-1$ coordinates equal to $0$. For $i > n-m$ there is no non-trivial solutions of \eqref{Eq:horBnTrans}, since $X L_0 \not = 0$  for $L_0\not = 0$. For each $i =1, \dots, n-m$ we can cross out first $i-1$ columns from $X$ and $A$ and then apply Proposition~\ref{Prop:GenXASpecificX} to the new $m \times (n-i+1)$ matrices. Denote $j = n - m - i + 1$. We get $j$ polynomial solutions of  \eqref{Eq:horBnTrans} corresponding to horizontal indices \[ \underbrace{q_j+1, \dots, q_j+1}_{j-r_j}, \qquad \underbrace{q_j+2, \dots, q_j+2}_{r_j},\] as required. We can apply  Proposition~\ref{L:NumberHorizontInd}, since the vectors in polynomial solutions are linearly independant: solutions for different $i$ belong to different columns in $\operatorname{b}(n)$, and for the same number of column $i$ they are linearly independant by construction.

\paragraph{Calculation of vertical indices.} Now, let us describe bases for vertical Kronecker blocks. For vertical blocks it is more convenient to consider the dual map \[ R_X^*: \operatorname{Mat}^*_{n \times m} \to \operatorname{b}^*(n).\] Notice that we study a restriction of a linear map $\gl(n) \to \operatorname{Mat}_{n \times m}$ to a matrix subalgebra $\mathfrak{h}= \operatorname{b}(n)$. Thus the dual map $\operatorname{Mat}^*_{n \times m} \to \mathfrak{h}^*$ is the composition of the dual map \[ \operatorname{Mat}^*_{n \times m} \to \gl(n)^* \approx \operatorname{Mat}^*_{n \times n}\] and the natural projection \[\pi: \operatorname{Mat}^*_{n \times n} \to \mathfrak{h}^*.\] We have the following trivial statement.

\begin{assertion}  
Let $\pi: \operatorname{Mat}^*_{n \times n} \to \operatorname{b}(n)^*$. Identify $\operatorname{Mat}_{n \times n}^*$ with $\operatorname{Mat}_{n \times n}$  by the formula \eqref{Eq:TrMatProd}. Then $\pi(A) = 0$ iff $A$ is strictly upper triangular. Simply speaking, $\operatorname{tr} (AU) = 0$ for all upper triangular $U$ if and only if $A$ is strictly upper triangular. \end{assertion}

Denote by $\operatorname{N}(n)$ the set of strictly upper triangular matrices. When we consider a dual representation the vertical blocks become horizontal. Thus, in order to apply Proposition~\ref{L:NumberHorizontInd} we search for $M_0, M_1 \dots, M_j \in \operatorname{Mat}_{m \times n} $  such that $M_0, M_1 \dots, M_j \in \operatorname{Mat}_{m \times n} $  such that \begin{equation} \label{Eq:BorelVertEqChain} \left(X - \lambda A \right) \left( M_0 + \lambda M_1 + \dots \lambda^j M_j \right) \in \operatorname{N}(n). \end{equation}
Here \eqref{Eq:BorelVertEqChain} means that all the matrices $XM_0$, $XM_i - A M_{i-1}$ for $i=1, \dots, j-1$ and $AM_j $ are strictly upper triangular. There are several cases.

\begin{enumerate}

\item \textit{Case $m = n$.} We can assume that $X = I_n$. Note that we can solve \eqref{Eq:BorelVertEqChain} independantly for each columns. Consider \eqref{Eq:BorelVertEqChain} for the $(k+1)$-th column (for $k \geq 0$). Note that in the $(k+1)$-th column of strictly upper triangular matrices last $n-k$ coordinates are equal to $0$. Denote $U^k = \operatorname{Span}\left(e_1, \dots, e_k \right)$. Then we seach for vectors $v_0, \dots, v_j$ such that \begin{equation} \label{Eq:BorelKronEqVectorMisN} (I_n -\lambda A) (v_0 + \lambda v_1 + \dots + \lambda^j v_j ) \in U^k.\end{equation} We start with the following trivial statement.

\begin{proposition}  Let $U \subset V$ be a linear subspace. Then for a generic linear map $A: V \to V$ we have \begin{equation} \label{Eq:DimSumSubset} \dim \left( U + A(U) + \dots + A^q(U) \right) = \min (q \cdot \dim U, \dim V). \end{equation} \end{proposition}

\begin{corollary} \label{Cor:GenericImageBasis} If \eqref{Eq:DimSumSubset} is satisfied, then we can choose a basis $e_1, \dots, e_k \in U$ and put $Ae_i = e_{i+k}$ for $i=k+1, \dots, n-k$ so that the vectors $e_1, \dots, e_n$ form a basis of $V$.  In this basis the matrix of $A$ has the form \[  A = \left(\begin{matrix} 0 & C \\ I_{n -k} & D \end{matrix} \right), \qquad C \in \operatorname{GL}(k).\] \end{corollary} 

The vectors $e_i$ from Corollary~\ref{Cor:GenericImageBasis} satisfy the following conditions: \[ e_i \in U, \qquad e_{i+ q k} - Ae_{i + (q-1)k} =0,\]   for $i =1, \dots, k$ and $i + q k \leq n$. Denote by $q_i$ the maximal $q$ such that $i + qk \leq n$. The solutions of \eqref{Eq:BorelKronEqVectorMisN} satisfy \begin{equation} \label{Eq:VectFinalCond} e_i \in U, \qquad e_{i+ q k} - Ae_{i + (q-1)k} \in U, \qquad Ae_{i+q_i k} \in U. \end{equation} It is not hard to add to vectors $e_{i + qk}$ some vectors from $U$ so that they would satisfy \eqref{Eq:VectFinalCond}. Thus, for the $(k+1)$-th column we got $k$ vertical indices corresponding to the ``chains'' \eqref{Eq:VectFinalCond}. If $n = q k + r$ for $0 \leq r < k$, then we got vertical indices \[ \underbrace{q, \dots, q}_{k-r}, \qquad \underbrace{q+1, \dots, q+1}_{r}.\] The first column ($k=0$) in strictly upper triangular matrices is trivial, hence we don't get indices for it. Taking all these indices for $k =1, \dots, n-1$ we get the required $\frac{n(n-1)}{2}$ vertical indices. We can apply  Proposition~\ref{Prop:GenXASpecificX}, since the vectors in polynomial solutions are linearly independant: solutions for different $k$ belong to different columns in $\operatorname{Mat}_{n \times n}$, and for the same number of column $k$ they are linearly independant by construction.

\item \textit{Case $m < n$.} We can assume that $X=\left( \begin{matrix} 0 \\ I_m \end{matrix} \right)$,  $A=\left( \begin{matrix} A_1 \\ A_2 \end{matrix} \right)$. Then we can take  $M_i = \left(\begin{matrix}  0  & \hat{M}_i \end{matrix} \right)$. It is easy to see that $M_i$ satisfy $\eqref{Eq:BorelVertEqChain}$ if and only if \[ 
\left( I_m - \lambda A_2\right) \left( \hat{M}_0 + \lambda \hat{M}_1 + \dots \lambda^j \hat{M}_j \right)  \in \operatorname{N}(n). \] We reduced the case $m < n$ to the previous case $m =n$. 

\item \textit{Case $m > n$.} We can assume that $ X=\left( \begin{matrix} I_n & 0\end{matrix} \right)$ and $A= \left( \begin{matrix}  A_1 & A_2 \end{matrix} \right)$. Consider equation \eqref{Eq:BorelVertEqChain} for the $(n-j+1)$-th columns (here $j= 1, \dots, n$). Denote the $(n-j+1)$-th column of $M_i$ by $v_i$. For stricly upper triangular matrices in the $(n-j+1)$-th column the last $j$ values are equal to $0$. Denote $U = \operatorname{Span}(e_1, \dots, e_{n-j+1})$. Then for the $(n-j+1)$-th columns \eqref{Eq:BorelVertEqChain} takes the form \[Xv_0 \in U, \qquad X v_{i+1} - Av_{i} \in U, \qquad Av_j \in U. \] Cross out first $(n-j+1)$-th rows in $X$ and $A$. For the new $j \times m$ matrices $\hat{X}$ and $\hat{A}$ we get the following equaitons \[ \hat{X} v_0 = 0, \qquad \hat{X} v_{i+1} - \hat{A}v_{i} =0, \qquad \hat{A}v_j =0.\] The rest of the proof is similar to the case of horizontal indices before. We can apply a statement similar to Proposition~\ref{Prop:GenXASpecificX}  for the matrix $\hat{X} = \left( \begin{matrix} I_j & 0\end{matrix} \right)$ and get the required vertical indices  \[ \underbrace{q_j, \dots, q_j}_{m-j - r_j}, \qquad \underbrace{q_j+1, \dots, q_j+1}_{r_j},\] where $m= q_j (m-j) + r_j, 0 \leq r_j <m-j$ for each $j = 1, \dots, n$.

\end{enumerate}

Theorem~\ref{T:JKSumStandardBn} is proved.

\subsection{Standard representation of $\operatorname{n}(n)$ }

For $\operatorname{n}(n)$  we consider only its standard representation, not the sums. 

\begin{theorem} \label{T:JKSumStandardNn} Let $\rho:\operatorname{n}(n) \to \operatorname{gl}(n)$ be the standard representation of the Lie algebra of strictly upper triangular matrices $\operatorname{n}(n)$. Then the JK invariants of $\rho$ are

 \begin{itemize}

\item  $\frac{(n-1)(n-2)}{2}$ horizontal indices: \[ \underbrace{1, \dots, 1 }_{\frac{(n-3)(n-2)}{2}}, \quad \underbrace{2, \dots, 2}_{n-2}. \] 

\item one Jordan $1 \times 1$ block, 
  
\item one vertical index $v_1 = 1$.
 
 \end{itemize}
 
\end{theorem}

\begin{remark} Similar to Remark~\ref{Rem:SumBlocksGL} we can check that the sum of sizes of blocks is \[\operatorname{dim} V \times \operatorname{dim} \mathfrak{g} = n \times \frac{n(n-1)}{2}.\] Indeed, 
\[\frac{(n-3)(n-2)}{2} \left[ 0 \times 1\right] + (n-2) \left[ 1 \times 2 \right] + 1 \times 1 + 1 \times 0 = n \times \frac{n(n-1)}{2}. \] The required ``identity'' is satisfied. \end{remark}

Denote the basis of $V^n$ by $e_1, \dots, e_n$. 

\begin{proposition} \label{Prop:Nn_SingStOreg} For the standard representation of $\operatorname{n}(n)$ we have \[\Sing = \operatorname{Span}\left(e_1, \dots, e_{n-1} \right). \] The numbers of horizontal and vertical indices are \[ \dimSt = \frac{(n-1)(n-2)}{2}, \qquad \codimO = 1 \] respectively.\end{proposition}

\begin{proof}[Proof of Proposition~\ref{Prop:Nn_SingStOreg}]  For a vector $x \in \operatorname{Span}\left(e_1, \dots, e_{k} \right)$ we have  \[ \dim \St_x = \frac{n (n-1)}{2}  - (k-1) \] The rest easily follows.  Proposition~\ref{Prop:Nn_SingStOreg} is proved. \end{proof}

Since a generic line $X+ \lambda A$ intersects the hyperplane $\operatorname{Span}\left(e_1, \dots, e_{n-1} \right)$ in one point, we get the following.

\begin{corollary} \label{Cor:Nn}  For the standard representation of $\operatorname{n}(n)$ the JK invariants (apart from horizontal and vertical indices) contain only one Jordan $1 \times 1$ block. \end{corollary}

In the next statement we identify vectors $\sum_i x^i e_i$ with $n\times 1$ matrices $\left(\begin{matrix} x^1 \\ \vdots \\ x^n \end{matrix}  \right)$.

\begin{proposition} \label{L:GenMatPairNn} Consider the standard representation for the Lie algebra  of strictly upper triangular matrices $\operatorname{n}(n)$.  Then the pencil $R_{X +\lambda A}$, where $X = e_n$ and $A = e_{n-1}$ is generic. \end{proposition}

\begin{proof}[Proof of Proposition~\ref{L:GenMatPairNn}] The Lie algebra $\operatorname{n}(n)$ can be invariantly defined through the standard flag \[ 0 \subset \operatorname{Span} (e_{1}) \subset \dots \subset \operatorname{Span} (e_{1},\ldots ,e_{n}) = V \] as operators $A$ such that $A(V_k) \subset V_{k-1}$, $V_k = \operatorname{Span} (e_{1},\ldots, e_{k})$. Therefore, the JK invariants for $\operatorname{n}(n)$ do not change under a change of coordinates that preserves the standard flag. Simply speaking, we can replace the basis $e_1, \dots, e_n$ with any basis \[ e_i' = c_{ii} e_i + \sum_{j < i}  c_{ij} e_j, \qquad c_{ii} \not = 0, \qquad i = 1,\dots, n. \] Since $X, A$ are generic, we can assume that $X$ is a regular element, i.e. $X \not \in \operatorname{Span}\left(e_1, \dots, e_{n-1} \right)$. Therefore we can take $X$ as the new last element of the basis $X = e_n$. The JK invariants for any two pairs $X, A$ and $X, A + \mu X$ coincide. Changing $A$ if nessesery we can assume that $A \in \Sing$. A generic $A \in \Sing$ has the form \[ A = c_{n-1} e_{n-1} + \sum_{j < n-1}  c_{j} e_j, \qquad c_{n-1} \not = 0. \] Thus we can take $A$ as another element of the basis $A = e_{n-1}$.  Proposition~\ref{L:GenMatPairNn} is proved. \end{proof}

\begin{proposition} \label{Prop:Kron1Nn}
Consider a generic pencil $R_{X }+\lambda R_A$ for the  standard representation of $\operatorname{n}(n)$. The number of horizontal and vertical indices equal to $1$  is
\[\dim \left(\operatorname{St}_X \cap \operatorname{St}_A\right) = \frac{(n-3)(n-2)}{2}, \qquad \codim \left(\operatorname{Im} R_X + \operatorname{Im} R_A \right) =   1 \] respectively.\end{proposition}

\begin{proof}[Proof of Proposition~\ref{Prop:Kron1Nn}] It is an easy calculations for $(X, A)$ from Proposition~\ref{L:GenMatPairNn}. Here $\operatorname{St}_X \cap \operatorname{St}_A$ are matrices with zeroes in the two right columns. And \[\operatorname{Im} R_X + \operatorname{Im} R_A =\operatorname{Span}\left(e_1, \dots, e_{n-1} \right),\] i.e. vectors with the last coordinate equal to zero. Proposition~\ref{Prop:Kron1Nn} is proved. \end{proof}

\begin{proof}[Proof of Theorem~\ref{T:JKSumStandardNn}] It follows from Propositions~\ref{Prop:Nn_SingStOreg} and \ref{Prop:Kron1Nn} and considerations of dimensions. Theorem~\ref{T:JKSumStandardNn} is proved. \end{proof}

\begin{remark} It is not hard to describe an explicit bases of all blocks for $X=e_{n}, A = e_{n-1}$. Here by $E_{ij}$ we denote a matrix with a 1 in the j-th column of the i-th row, and 0s everywhere else.

\begin{itemize}

\item The horizontal indices equal to 1 correspond to the matrices \[E_{ij}, \qquad 1 < i < j \leq n-2,\]  since $E_{ij} e_{n-1} = E_{ij} e_{n} = 0$. 

\item The horizontal indices equal to $2$ correspond to matrices $E_{i, n-1}$ and $E_{i, n-1}$ in $\operatorname{n}(n)$ for $i \leq n-2$. Note the chain of identities \begin{gather*} E_{i, n-1} e_n =0, \\  E_{i, n-1}e_{n-1}=E_{i n} e_n ,  \\ E_{in}  e_{n-1} =0. \end{gather*} 

\item The Jordan $1\times 1$ block corresponds to the matrix  $E_{n-1, n} \in \operatorname{n}(n)$ since the images of $X, A$  are collinear: \[ E_{n-1, n} e_{n-1} =0, \qquad E_{n-1, n} e_n = e_{n-1}. \] 

\item The vertical index $v_1 = 1$ corresponds to the covector $e^n$ in the dual basis (i.e. the image of $R_{X + \lambda A}$ lies in $\operatorname{Ann} \left( e^n \right) =\operatorname{Span}\left(e_1, \dots, e_{n-1} \right) $).

  \end{itemize}

\end{remark}

\section{JK invariants for actions by congruence on symmetric and skew-symmetric forms} \label{S:CongAct}

In this section we consider the action of the group $\operatorname{GL}(V)$ by congruence on the space of bilinear forms on $V$: \[ P, Q \to P Q P^T. \] The corresponding Lie algebra representation \[ \rho: \gl(V) \to \gl(V \otimes V)\] is given by \[\rho( X) Q = XQ + Q X^{T}. \] In this section, we study the JK invariants for the corresponding representations of the Lie algebras $\mathfrak{g} = \gl(V)$ or $\operatorname{sl}(V)$ on the spaces of symmetric forms $S^2(V)$ and skew-symmetric forms $\Lambda^2(V)$.

\subsection{Action of $\operatorname{GL}(n)$ on symmetric forms}  

\subsubsection{JK invariants} 

\begin{theorem} \label{Th:GL_Sym_CongAct} Let $\rho$ be the representation of $\gl(V)$ on the space of symmetric forms $S^2(V)$ corresponding to the congruence action. Then the JK invariants of $\rho$ are \begin{itemize}

\item $n$ Jordan $1 \times 1$ blocks with different eigenvalues,

\item $\frac{n(n-1)}{2}$ horizontal indices $h_i = 2$.

\end{itemize} 

\end{theorem}

First, let us find the number of blocks. It is well-known that any complex symmetric form $Q$ with $\rk Q = k$  is congruent to \[Q = \operatorname{diag} (\underbrace{1, \dots, 1}_k, 0 \dots 0).\] The stabilizer is given by the condition \[X Q = - \left( X Q\right)^T,\] hence \[ \dim \St_Q = \frac{k(k-1)}{2}  + n(n-k). \] We get the following proposition.

\begin{proposition}  \label{Prop:GL_SymCong} Let $\rho$ be the representation of $\gl(V)$ on the space of symmetric forms $S^2(V)$ corresponding to the congruence action. 

\begin{enumerate}

\item The singuar set is \[ \Sing = \left\{ \left. Q \right| \quad  \det Q =0 \right\} . \] 

\item Hence, the JK invariants contain Jordan blocks with $n$ different eigenvalues.  

\item The numbers of horizontal and vertical indices are \[ \dimSt = \frac{n (n-1)}{2}, \qquad \codimO = 0 \] respectively.

\end{enumerate}

\end{proposition}

\subsubsection{Generic pairs of forms} 

Since we calculate the JK invariants for generic pencils, we want to know what pairs of symmetric forms are generic. Luckily there is a canonical form for pairs of symmetric forms, described in  \cite{Gantmaher88} (see also \cite{Thompson}). We reformulate this result similar to the Jordan-Kronecker theorem below.

\begin{theorem}[Jordan-Kronecker theorem for symmetric forms, \cite{Gantmaher88}] \label{T:Jordan-Kronecker_theorem_Sym}
Let $A$ and $B$ be symmetric bilinear forms on a
finite-dimension vector space $V$ over a field $\mathbb{K}$ with $\operatorname{char} \mathbb{K} = 0$. If the field $\mathbb{K}$
is algebraically closed, then there exists a basis of the space $V$
such that the matrices of both forms $A$ and $B$ are block-diagonal
matrices:

\[
A =
\begin{pmatrix}
A_1 &     &        &      \\
    & A_2 &        &      \\
    &     & \ddots &      \\
    &     &        & A_k  \\
\end{pmatrix},
\qquad  B=
\begin{pmatrix}
B_1 &     &        &      \\
    & B_2 &        &      \\
    &     & \ddots &      \\
    &     &        & B_k  \\
\end{pmatrix}
\]

where each pair of corresponding blocks $A_i$ and $B_i$ is one of
the following:

\begin{itemize}

\item (symmetric) Jordan block with eigenvalue $\lambda_i \in \mathbb{K}$:   \[A_i =\left(
\begin{matrix}
    &&        & \lambda_i \\
      & &  \lambda_i  & 1     \\
      &    \udots     & \udots &   \\
    \lambda_i    &   1     &        &  \\
    \end{matrix} \right),
\qquad  B_i= \left( \begin{matrix}
     & &        & 1\\
      &  &  \udots    &     \\
      &   \udots      & &   \\
    1   &        &        &   \\
    \end{matrix} 
 \right)
\] \item (symmetric) Jordan block with eigenvalue $\infty$: \[A_i = \left( \begin{matrix}
     & &        & 1\\
      &  & \udots    &     \\
      &   \udots      & &   \\
    1   &        &        &   \\
    \end{matrix} 
 \right),
\qquad  B_i= \left(
\begin{matrix}
    &&        & 0 \\
      & & 0  & 1     \\
      &    \udots     & \udots &   \\
    0    &   1     &        &  \\
    \end{matrix} \right)
\] \item   (symmetric) Kronecker block:  \[ A_i = \left(
\begin{array}{c|c}
  0 & \begin{matrix}
   1 & 0      &        &     \\
      & \ddots & \ddots &     \\
      &        & 1    &  0  \\
    \end{matrix} \\
  \hline
  \begin{matrix}
  1  &        &    \\
  0   & \ddots &    \\
      & \ddots & 1 \\
      &        & 0  \\
  \end{matrix} & 0
 \end{array}
 \right), \qquad  B_i= \left(
\begin{array}{c|c}
  0 & \begin{matrix}
    0 & 1      &        &     \\
      & \ddots & \ddots &     \\
      &        &   0    & 1  \\
    \end{matrix} \\
  \hline
  \begin{matrix}
  0  &        &    \\
  1   & \ddots &    \\
      & \ddots & 0 \\
      &        & 1  \\
  \end{matrix} & 0
 \end{array}
 \right).
 \] 
 \end{itemize}

\end{theorem}

Now, in theory, one could calculate the JK invariants for any pair of forms $Q, B$. Yet it is much simpler to do it for generic pairs. 

\begin{corollary} \label{Cor:GenMatPair_GLSym} Any generic pair of symmetric forms $Q$ and $B$ is congruent to one of the following pairs: 
\begin{equation} \label{Eq:PrinAx} Q= \left( \begin{matrix}\lambda_1& & \\ &\ddots& \\ & &\lambda_n \end{matrix} \right), \qquad B= \left( \begin{matrix}1& & \\ &\ddots& \\ & &1 \end{matrix} \right), \end{equation} where $\lambda_i$ are distinct. \end{corollary}

\begin{proof}[Proof of Corollary~\ref{Cor:GenMatPair_GLSym} ] Since the pair $Q, B$ is generic, we can assume that both forms are non-degenerate. Thus there are no (symmetric) Kronecker blocks, since the number of such blocks is $\operatorname{Ker} (Q +\lambda B)$ for generic $\lambda$. Next, for generic  $Q, B$ the operator $P = B^{-1} Q$ is semisimple with distinct eigenvalues. Therefore all (symmetric) Jordan blocks are $1\times 1$ and all eigenvalues are distinct. Corollary~\ref{Cor:GenMatPair_GLSym}  is proved. \end{proof}

Calculating stabilizers and images for the pairs $(Q, B)$ from Corollary~\ref{Cor:GenMatPair_GLSym} we get the following.

\begin{proposition} 
\label{Prop:SymForm_GL_TrivKronBlocks}
Consider a generic pencil $R_{Q }+\lambda R_B$ for the representation of $\gl(V)$ on the space of symmetric forms $S^2(V)$ corresponding to the congruence action. Then the number of horizontal and vertical indices equal to $1$ is
\[ \dim \left(\operatorname{St}_Q \cap \operatorname{St}_B\right) = 0, \qquad  \codim \left(\operatorname{Im} R_Q + \operatorname{Im} R_B \right)  = 0\] respectively.\end{proposition}

\subsubsection{Proof of Theorem~\ref{Th:GL_Sym_CongAct}} 

\begin{proof}[Proof of Theorem~\ref{Th:GL_Sym_CongAct}] Theorem~\ref{Th:GL_Sym_CongAct}  follows from Propositions~\ref{Prop:GL_SymCong} and \ref{Prop:SymForm_GL_TrivKronBlocks} and considerations of dimensions.\end{proof}

\begin{remark}
For pairs $Q, B \in S^2(V)$ given by \eqref{Eq:PrinAx} we can explicitely describe bases of all blocks.

\begin{itemize} 

\item The Jordan $1 \times 1$ blocks correpsond to the matrices $E_{ii} \in \gl(n)$ and $E_{ii} \in S^2(V)$  \[ \rho\left(E_{ii}\right) Q = E_{ii} , \qquad \rho\left( E_{ii}\right) B = \lambda_i E_{ii} \] 

\item For the horizontal indices we take Kronecker blocks with bases \[ A_0  = E_{ij} - E_{ji}, \qquad A_1 = -\lambda_i E_{ij} + \lambda_j E_{ji} \] in $\gl(n)$, where $1 \leq i < j \leq n$. We get a Kronecker $1 \times 2$ block since 
\[ \rho\left(A_0 + \lambda A_1 \right) \left( Q - \lambda B\right) = 0.\]  
  \end{itemize}
  \end{remark}

\subsection{Action of $\operatorname{SL}(n)$ on symmetric forms} 
  
For  $\operatorname{SL}(n)$ we always assume that $n > 1$.

\begin{theorem} \label{T:SL_SymCong} Let $\rho$ be the representation of $\operatorname{sl}(V)$ on the space of symmetric forms $S^2(V)$ corresponding to the congruence action. Then the JK invariants of $\rho$ are \begin{itemize}

\item $\frac{n(n-1)}{2}$ horizontal indices $\mathrm{h}_i = 2$.

\item one vertical index $\mathrm{v}_1 = n$.

\end{itemize} 

\end{theorem}

$\operatorname{GL}(n)$ is the product of $\operatorname{SL}(n)$ with the subgroup of scalar matrices. Thus, roughly speaking, the congruence actions of $\operatorname{GL}(n)$ and $\operatorname{SL}(n)$ may differ only by a multiplication on a scalar matrix. Of course, unlike $\operatorname{GL}(n)$, the congruence action of $\operatorname{SL}(n)$ preserves the determinant of matrices. It is easy to prove the following.

\begin{proposition} \label{Prop:CongOrbSLSym}
For any complex symmetric form $Q$ with $\rk Q = k$  there exists $P \in \operatorname{SL}(n)$ such that \[P QP^{T} = \operatorname{diag} (\underbrace{\mu, \dots, \mu}_k, 0 \dots 0).\]  Thus, for the representation $\rho$ from Theorem~\ref{T:SL_SymCong}  if $Q \in S^2(V)$ and $\rk Q = k$, then the dimension of stabilizer is \begin{equation} \label{Eq:SLStabRank} \dim \St_Q = \begin{cases} \frac{k(k-1)}{2}  + n(n-k)-1, \qquad& \text{ if } k <n, \\ \frac{n(n-1)}{2} , \qquad &\text{ if } k =n.   \end{cases}  \end{equation}  \end{proposition}

Using Proposition~\ref{Prop:CongOrbSLSym} it is easy to calculate the number of Jordan and Kronecker blocks.

\begin{proposition}  \label{Prop:SL_SymCong} Let $\rho$ be the representation of  $\operatorname{sl}(V)$ on the space of symmetric forms $S^2(V)$ corresponding to the congruence action. 

\begin{enumerate}

\item The singuar set is \begin{equation} \label{Eq:SLsingSet} \Sing = \left\{ \left. Q \right| \quad  \rk Q < n-1  \right\} . \end{equation} 

\item Hence, there are no Jordan blocks in the JK invariants.  

\item The numbers of horizontal and vertical indices are \[ \dimSt = \frac{n (n-1)}{2},  \qquad \codimO = 0 \] respectively.
\end{enumerate}

\end{proposition}

\begin{proof}[Proof of Proposition~\ref{Prop:SL_SymCong}] The dimension $\operatorname{St}_Q$ with $\operatorname{rk} Q = k$ is given by \eqref{Eq:SLStabRank}. If $k = n-1$, then \[ \dim \operatorname{St}_Q = \frac{(n-1)(n-2)}{2} + n - 1 = \frac{n(n-1)}{2}.\] It is easy to see that for lesser $k$ the dimension of $\operatorname{St}_Q$ drops. Thus, the singular set $\Sing$ is \eqref{Eq:SLsingSet}. The rest of the proof is by direct calculation. Proposition~\ref{Prop:SL_SymCong} is proved. \end{proof}

The generic pairs $(Q, B)$ are similar to the case $\operatorname{gl}(n)$. Yet, the congruence action has to preserve determinants, thus the canonical form for $B$ is a scalar and not the identity matrix.

\begin{corollary} \label{Cor:GenMatPair_SLSym} Consider the congruence action of $\operatorname{SL}(n)$ on pairs of symmetric forms $Q, B \in S^2(V)$. Any generic pair is congruent to one of the following pairs: 
\begin{equation} \label{Eq:PrinAxSL} Q=\left( \begin{matrix}\lambda_1& & \\ &\ddots& \\ & &\lambda_n \end{matrix} \right), \qquad B=\left( \begin{matrix}\mu& & \\ &\ddots& \\ & &\mu \end{matrix} \right), \end{equation} where $\lambda_i$ and $\mu$ are distinct and non-zero. \end{corollary}

Calculating stabilizers and images for the pairs $(Q, B)$ from Corollary~\ref{Cor:GenMatPair_SLSym} we get the following.

\begin{proposition} 
\label{Prop:SymForm_SL_TrivKronBlocks}
Consider a generic pencil $R_{Q }+\lambda R_B$ for the representation of $\operatorname{sl}(V)$ on the space of symmetric forms $S^2(V)$ corresponding to the congruence action. Then the number of horizontal and vertical indices equal to $1$ is
\[ \dim \left(\operatorname{St}_Q \cap \operatorname{St}_B\right) = 0, \qquad  \codim \left(\operatorname{Im} R_Q + \operatorname{Im} R_B \right)  = 0\] respectively.\end{proposition}

All that remains is to find the number and sizes of vertical blocks. We do it using Theorem~\ref{thm3}. The next statement about the algebra of invariants is trivial.

\begin{proposition} \label{SL_SymCong_AlgInv} \label{Prop:SLn}  Consider the congruence action of $\operatorname{SL}(n)$ on symmetric forms $S^2(V)$. Then the algebra of invariants is freely generated by the polynomial $\operatorname{det} Q$ for $Q\in S^2(V)$. \end{proposition}

We have all the information to calculate the JK invariants for the differential of conguence action of $\operatorname{SL}(n)$.

\begin{proof}[Proof of Theorem~\ref{T:SL_SymCong} ] By Theorem~\ref{thm3} and Proposition~\ref{SL_SymCong_AlgInv} there is only 1 vertical index $\mathrm{v}_1 = n$. Now Theorem~\ref{T:SL_SymCong} follows from Propositions~\ref{Prop:SL_SymCong} and \ref{Prop:SymForm_SL_TrivKronBlocks} and considerations of dimensions.\end{proof}

\subsection{Action on skew-symmetric forms,  $\dim V = 2n$} 

Now, let study the congruence action of $\operatorname{GL}(V)$ and $\operatorname{SL}(V)$ on skew-symmetric forms $\Lambda^2(V)$. For skew-symmetric forms we always assume that $\operatorname{char}\mathbb{K} \not = 2$ (we are mostly interested in the case $\operatorname{char}\mathbb{K} = \mathbb{C}$). In general, the answers and the proofs are similar to the symmetric case. Although now the answers will depend on the parity of $\dim V$, since generic pairs of skew-symmetric forms $Q, B$ will be different. First, let us recall the theorem about the canonical form for a pair of skew-symmetric forms (see \cite{Thompson}, which is based on \cite{Gantmaher88}). 

\begin{theorem}[Jordan--Kronecker theorem, \cite{Thompson}]\label{T:Jordan-Kronecker_theorem}
Let $A$ and $B$ be skew-symmetric bilinear forms on a
finite-dimension vector space $V$ over a field $\mathbb{K}$ with $\textmd{char }  \mathbb{K} \ne 2$. If the field $\mathbb{K}$
is algebraically closed, then there exists a basis of the space $V$
such that the matrices of both forms $A$ and $B$ are block-diagonal
matrices:

\[
A =
\begin{pmatrix}
A_1 &     &        &      \\
    & A_2 &        &      \\
    &     & \ddots &      \\
    &     &        & A_k  \\
\end{pmatrix}
\qquad  B=
\begin{pmatrix}
B_1 &     &        &      \\
    & B_2 &        &      \\
    &     & \ddots &      \\
    &     &        & B_k  \\
\end{pmatrix}
\]

where each pair of corresponding blocks $A_i$ and $B_i$ is one of
the following:

\begin{itemize}

\item Jordan block with eigenvalue $\lambda_i \in \mathbb{K}$: {\scriptsize  \[A_i =\left(
\begin{array}{c|c}
  0 & \begin{matrix}
   \lambda_i &1&        & \\
      & \lambda_i & \ddots &     \\
      &        & \ddots & 1  \\
      &        &        & \lambda_i   \\
    \end{matrix} \\
  \hline
  \begin{matrix}
  \minus\lambda_i  &        &   & \\
  \minus1   & \minus\lambda_i &     &\\
      & \ddots & \ddots &  \\
      &        & \minus1   & \minus\lambda_i \\
  \end{matrix} & 0
 \end{array}
 \right),
\quad  B_i= \left(
\begin{array}{c|c}
  0 & \begin{matrix}
    1 & &        & \\
      & 1 &  &     \\
      &        & \ddots &   \\
      &        &        & 1   \\
    \end{matrix} \\
  \hline
  \begin{matrix}
  \minus1  &        &   & \\
     & \minus1 &     &\\
      &  & \ddots &  \\
      &        &    & \minus1 \\
  \end{matrix} & 0
 \end{array}
 \right)
\]} \item Jordan block with eigenvalue $\infty$ {\scriptsize \[
A_i = \left(
\begin{array}{c|c}
  0 & \begin{matrix}
   1 & &        & \\
      &1 &  &     \\
      &        & \ddots &   \\
      &        &        & 1   \\
    \end{matrix} \\
  \hline
  \begin{matrix}
  \minus1  &        &   & \\
     & \minus1 &     &\\
      &  & \ddots &  \\
      &        &    & \minus1 \\
  \end{matrix} & 0
 \end{array}
 \right),
\quad B_i = \left(
\begin{array}{c|c}
  0 & \begin{matrix}
    0 & 1&        & \\
      & 0 & \ddots &     \\
      &        & \ddots & 1  \\
      &        &        & 0   \\
    \end{matrix} \\
  \hline
  \begin{matrix}
  0  &        &   & \\
  \minus1   & 0 &     &\\
      & \ddots & \ddots &  \\
      &        & \minus1   & 0 \\
  \end{matrix} & 0
 \end{array}
 \right)
 \] } \item   Kronecker block   {\scriptsize \[ A_i = \left(
\begin{array}{c|c}
  0 & \begin{matrix}
   1 & 0      &        &     \\
      & \ddots & \ddots &     \\
      &        & 1    &  0  \\
    \end{matrix} \\
  \hline
  \begin{matrix}
  \minus1  &        &    \\
  0   & \ddots &    \\
      & \ddots & \minus1 \\
      &        & 0  \\
  \end{matrix} & 0
 \end{array}
 \right), \quad  B_i= \left(
\begin{array}{c|c}
  0 & \begin{matrix}
    0 & 1      &        &     \\
      & \ddots & \ddots &     \\
      &        &   0    & 1  \\
    \end{matrix} \\
  \hline
  \begin{matrix}
  0  &        &    \\
  \minus1   & \ddots &    \\
      & \ddots & 0 \\
      &        & \minus1  \\
  \end{matrix} & 0
 \end{array}
 \right)
 \] }
 \end{itemize}

\end{theorem}

Each Kronecker block is a $(2k_i-1) \times (2k_i-1)$ block, where
$k_i \in \mathbb{N}$. If $k_i=1$, then the blocks are $1\times 1$
zero matrices
\[
A_i =
\begin{pmatrix}
0
\end{pmatrix} \quad  B_i=
\begin{pmatrix}
0
\end{pmatrix}
\]

\subsubsection{Congruence action of $\operatorname{GL}(2n)$} 

\begin{theorem} \label{T:GL_SkewEvenCong} Let $\rho$ be the representation of $\gl(2n)$ on the space of skew-symmetric forms $\Lambda^2(V)$, where $\dim V = 2n$, corresponding to the congruence action. Then the JK invariants of $\rho$ are \begin{itemize}

\item $n(2n+1)$ horizontal indices  \[ \underbrace{1, \dots, 1 }_{3n }, \qquad  \underbrace{2, \dots, 2 }_{2n(n-1)}, \] 

\item $n$ Jordan $1\times 1$ blocks with distinct eigenvalues.

\end{itemize} 

\end{theorem}

A complex skew-symmetric form $Q$ with $\rk Q = 2k$  is congruent  to  \[Q= \left( \begin{matrix}  0 & 0 & 0 \\ 0 & 0  & I_k \\ 0 & - I_k & 0 \end{matrix} \right). \] The dimension of stabilizer is \begin{equation} \label{Eq:DimStSkewGL} \dim \St_Q =  k(2k+1) + \dim V (\dim V -2k).   \end{equation}

\begin{proposition}  \label{Prop:SL_SkewEvenCong} Let $\rho$ be the representation of $\gl(V)$ on the space of skew-symmetric forms $\Lambda^2(V)$, where $\dim V = 2n$, corresponding to the congruence action. 

\begin{enumerate}

\item The singuar set is \[ \Sing = \left\{ \left. Q \right| \quad  \operatorname{Pf} Q =0\right\},  \] where $\operatorname{Pf} Q$ is the pfaffian of matrix. 

\item Since $\Sing$ is defined by one polynomial of degree $n$,  the JK invariants contain Jordan blocks with $n$  distinct eigenvalues.  

\item The numbers of horizontal and vertical indices are \[ \dimSt = n(2n+1),  \qquad \codimO = 0 \] respectively.

\end{enumerate}
\end{proposition}

In theory, one could calculate the JK invariants for any pair of forms $Q, B$. Yet it is much simpler to do it for generic pairs. 

\begin{corollary} \label{Cor:GenMatPair_GLSkewEven} Any generic pair of skew-symmetric forms $Q$ and $B$ on an even-dimensional space is congruent to one of the following pairs: 
\begin{equation} \label{Eq:JKFormsEven}  Q = \left( {\displaystyle {\begin{matrix}{\begin{matrix}0&\lambda_1\\-\lambda_1&0\end{matrix}}&& \\&\ddots &\\ &&{\begin{matrix}0&\lambda_n\\-\lambda_n &0\end{matrix}}\end{matrix}}}\right), \qquad B  \left( {\displaystyle {\begin{matrix}{\begin{matrix}0&1\\-1&0\end{matrix}}&& \\&\ddots &\\ &&{\begin{matrix}0&1\\-1&0\end{matrix}}\end{matrix}}}\right),\end{equation}  where $\lambda_i$ are distinct. \end{corollary}

\begin{proof}[Proof of Corollary~\ref{Cor:GenMatPair_GLSkewEven} ] Since the pair $Q, B$ is generic, we can assume that both forms are non-degenerate. Thus there are no (skew-symmetric) Kronecker blocks, since the number of such blocks is $\operatorname{Ker} (Q +\lambda B)$ for generic $\lambda$. Next, for generic  $Q, B$ the operator $P = B^{-1} Q$ is semisimple. Since $Q, B$ are skew-symmetric, the eigenlvaues of $P$ come in pairs. But for generic $Q, B$ these pairs are distinct. Therefore all (skew-symmetric) Jordan blocks are $2\times 2$ and all eigenvalues are distinct. Corollary~\ref{Cor:GenMatPair_GLSkewEven}   is proved. \end{proof}

Calculating the stabilizers for forms from Corollary~\ref{Cor:GenMatPair_GLSkewEven} we get the following result. 

\begin{proposition} 
\label{Prop:SkewEvenForm_GL_TrivKronBlocks}
Consider a generic pencil $R_{Q }+\lambda R_B$ for the representation of $\operatorname{sl}(V)$ on the space of skew-symmetric forms $\Lambda^2(V)$, with $\dim V = 2n$, corresponding to the congruence action. Then the number of horizontal and vertical indices equal to $1$ is
\[ \dim \left(\operatorname{St}_Q \cap \operatorname{St}_B\right) = 3n, \qquad  \codim \left(\operatorname{Im} R_Q + \operatorname{Im} R_B \right)  = 0\] respectively.\end{proposition}

\begin{proof}[Proof of Proposition~\ref{Prop:SkewEvenForm_GL_TrivKronBlocks}] Elements $X \in \operatorname{St}_Q \cap \operatorname{St}_B$ are given by \[ XQ + QX^T = 0, \qquad XA + AX^T = 0. \]  Consider the group of linear automorphisms $\operatorname{Aut}(V, Q, B)$ that preserve the forms $Q$ and $B$: \[ Q(Cu, Cv) = Q(u, v), \qquad B(Cu, Cv) = B(u, v).\] Its Lie algebra $\operatorname{aut}(V, Q, B)$ is given by the equations \[ CQ + QC^T = 0, \qquad CA + AC^T = 0. \] It is easy to see that $C \in \operatorname{aut}(V, Q, B)$ if and only if $C^T \in \operatorname{St}_Q \cap \operatorname{St}_B$. The Lie algebra $\operatorname{aut}(V, Q, B)$ was described in \cite{Pumei14}. In this particular case $\operatorname{aut}(V, Q, B)$ is quite simple. Linear operators that preserve $Q$ and $B$ also preserve the recursion operator $P = B^{-1}Q$. Thus $X \in \operatorname{St}_Q \cap \operatorname{St}_B$ have block-diagonal structure \[ X = \left( \begin{matrix} X_1 & & \\ & \ddots & \\  & & X_n\end{matrix} \right),\] where $X_j$ are $2\times 2$ blocks. It is easy to check that $X \in \operatorname{St}_Q \cap \operatorname{St}_B$  iff all $X_i \in \operatorname{sp}(2)$. Thus, $\dim \left(\operatorname{St}_Q \cap \operatorname{St}_B\right) = 3n$. Finally, $\codim \left(\operatorname{Im} R_Q + \operatorname{Im} R_B \right)  = 0$ since $\codimO = 0$. Proposition~\ref{Prop:SkewEvenForm_GL_TrivKronBlocks} is proved. \end{proof}

Now we can calculate the JK invariants for the differential of the conguence action of $\operatorname{GL}(2n)$ on skew-symmetric forms $\Lambda^2(V^{2n})$.

\begin{proof}[Proof of Theorem~\ref{Prop:SL_SkewEvenCong}] Theorem~\ref{Prop:SL_SkewEvenCong}  follows from Propositions~\ref{Prop:SL_SkewEvenCong} and \ref{Prop:SkewEvenForm_GL_TrivKronBlocks} and considerations of dimensions.\end{proof}

\subsubsection{Congruence action of $\operatorname{SL}(2n)$} 

\begin{theorem} \label{T:SL_SkewEvenCong} Let $\rho$ be the representation of $\operatorname{sl}(2n)$ on the space of skew-symmetric forms $\Lambda^2(V)$, where $\dim V = 2n$, corresponding to the congruence action. Then the JK invariants of $\rho$ are \begin{itemize}

\item $n(2n+1)$ horizontal indices  \[ \underbrace{1, \dots, 1 }_{3n }, \qquad  \underbrace{2, \dots, 2 }_{2n(n-1)}. \] 

\item one vertical index $\mathrm{v}_1 = n$. 

\end{itemize} 

\end{theorem}

The main difference with $\operatorname{gl}(V)$ is that $\operatorname{SL}(V)$ preserves the determinant of matrices. Hence a complex skew-symmetric form $Q$ with $\rk Q = 2k$  is congruent $Q \to P QP^T$ with $P \in \operatorname{SL}(V)$ to \[Q= \left( \begin{matrix}  0 & 0 & 0 \\ 0 & 0  & \mu I_k \\ 0 & - \mu I_k & 0 \end{matrix} \right) \] for some $\mu \in \mathbb{K}$.  The dimension of stabilizer is \begin{equation} \label{Eq:SLStabRankSkewEven} \dim \St_Q = \begin{cases}  k(2k+1) + \dim V(\dim V-2k)-1, \qquad & \text{ if } 2k < \dim V \\ \frac{1}{2} \dim V(\dim V+1) \qquad & \text{ if } 2k=\dim V. \end{cases}  \end{equation}

\begin{proposition}  \label{Prop:SL_SkewEvenCong} Let $\rho$ be the representation of $\operatorname{sl}(V)$ on the space of skew-symmetric forms $\Lambda^2(V)$, where $\dim V = 2n$, corresponding to the congruence action.

\begin{enumerate}

\item The singuar set is \begin{equation} \label{Eq:SLsingSetSkewEven} \Sing = \left\{ \left. Q \right| \quad  \operatorname{rk} Q < 2n-2\right\}. \end{equation} 
\item Since $\codim \Sing > 1$ there are no Jordan blocks in the JK invariants.  
\item The numbers of horizontal and vertical indices are \[ \dimSt = n(2n+1),  \qquad \codimO = 1 \] respectively.\end{enumerate}

\end{proposition}

\begin{proof}[Proof of Proposition~\ref{Prop:SL_SkewEvenCong}]  The dimension $\operatorname{St}_Q$ with $\operatorname{rk} Q = 2k$ is given by \eqref{Eq:SLStabRankSkewEven}. If $k = n-1$, then \[ \dim \operatorname{St}_Q = (n-1)(2n-1) + 4n - 1 = n(2n+1).\] It is easy to see that for lesser $k$ the dimension of $\operatorname{St}_Q$ drops. Thus, the singular set $\Sing$ is \eqref{Eq:SLsingSetSkewEven}. The rest of the proof is by direct calculations. Proposition~\ref{Prop:SL_SkewEvenCong} is proved. \end{proof}

In theory, one could calculate the JK invariants for any pair of forms $Q, B$. Yet it is much simpler to do it for generic pairs. It is not hard to show that a generic pair for $\operatorname{SL}(2n)$ differs from the generic pair for $\operatorname{GL}(2n)$ by a multiplication on a scalar matrix.

\begin{corollary} \label{Cor:GenMatPair_SLSkewEven} Consider the congruence action of $\operatorname{SL}(2n)$ on pairs of skew-symmetric forms $Q, B \in S^2(V^{2n})$. Any generic pair is congruent to some of the following pairs: 
\[  Q = \left( {\displaystyle {\begin{matrix}{\begin{matrix}0&\lambda_1\\-\lambda_1&0\end{matrix}}&& \\&\ddots &\\ &&{\begin{matrix}0&\lambda_n\\-\lambda_n &0\end{matrix}}\end{matrix}}}\right), \qquad B = \left( {\displaystyle {\begin{matrix}{\begin{matrix}0&\mu\\-\mu&0\end{matrix}}&& \\&\ddots &\\ &&{\begin{matrix}0&\mu\\-\mu&0\end{matrix}}\end{matrix}}}\right),\]  where $\lambda_i$ and $\mu$ are distinct and non-zero. \end{corollary}

Calculating the stabilizers for forms from Corollary~\ref{Cor:GenMatPair_GLSkewEven} we get the following result. 

\begin{proposition} 
\label{Prop:SkewEvenForm_SL_TrivKronBlocks}
Consider a generic pencil $R_{Q }+\lambda R_B$ for the representation of $\operatorname{sl}(V)$ on the space of skew-symmetric forms $\Lambda^2(V)$, with $\dim V = 2n$, corresponding to the congruence action. Then the number of horizontal and vertical indices equal to $1$ is
\[ \dim \left(\operatorname{St}_Q \cap \operatorname{St}_B\right) = 3n, \qquad  \codim \left(\operatorname{Im} R_Q + \operatorname{Im} R_B \right)  = 0\] respectively.\end{proposition}

\begin{proof}[Proof of Proposition~\ref{Prop:SkewEvenForm_SL_TrivKronBlocks}]
\begin{enumerate}

\item Similar to Proposition~\ref{Prop:SkewEvenForm_GL_TrivKronBlocks} $ \operatorname{St}_Q \cap \operatorname{St}_B$ is the sum of $n$ Lie algebras $\operatorname{sp}(2)$, corresponding to each $2 \times 2$ block. Thus $\dim \left(\operatorname{St}_Q \cap \operatorname{St}_B\right) = 3n$. 

\item $\codimO = 1$ since the $\operatorname{SL}(n)$ action preserves the determinant of matrices. $\det(Q + \lambda B)$ is not a constant for generic $Q, B$. Thus $\codim \left(\operatorname{Im} R_Q + \operatorname{Im} R_B \right)  = 0$.

\end{enumerate}

Proposition~\ref{Prop:SkewEvenForm_SL_TrivKronBlocks} is proved.

\end{proof}

All that remains is to find the number and sizes of vertical blocks. We do it using Theorem~\ref{thm3}. The next statement about the algebra of invariants is trivial.

\begin{proposition} \label{SL_SkewEvenCong_AlgInv} \label{Prop:SLn} Consider the congruence action of $\operatorname{SL}(2n)$ on skew-symmetric forms $\Lambda^2(V^{2n})$. Then the algebra of invariants is freely generated by the pfaffian $\operatorname{Pf} Q$ for $Q \in \Lambda^2(V^{2n})$. \end{proposition}

We have all the information to calculate the JK invariants for the differential of conguence action of $\operatorname{SL}(n)$ on skew-symmetric forms $\Lambda^2(V^{2n})$.

\begin{proof}[Proof of Theorem~\ref{T:SL_SkewEvenCong} ] By Theorem~\ref{thm3} and Proposition~\ref{SL_SkewEvenCong_AlgInv} there is only 1 vertical index $\mathrm{v}_1 = n$. Now Theorem~\ref{T:SL_SkewEvenCong} follows from Propositions~\ref{Prop:SL_SkewEvenCong} and \ref{Prop:SkewEvenForm_SL_TrivKronBlocks} and considerations of dimensions.\end{proof}

\subsection{Action on skew-symmetric forms,  $\dim V = 2n+1$} 

The cases of $\operatorname{gl}(2n+1)$ and $\operatorname{sl}(2n+1)$  are almost identical. Let \begin{equation} \label{Eq:DeltaGLSL} \delta = \begin{cases} 1, \qquad \text{ for } \operatorname{gl}(n), \\ 0, \qquad \text{ for } \operatorname{sl}(n). \end{cases}\end{equation}

\begin{theorem} \label{T:GL_SkewOddCong} Let $\rho$ be the representation of $\gl(2n+1)$ or $\operatorname{sl}(2n+1)$ on the space of skew-symmetric forms $\Lambda^2(V)$, where $\dim V = 2n+1$, corresponding to the congruence action. Then the JK invariants of $\rho$ are $n(2n+3) + \delta$ horizontal indices  
 \[ \underbrace{1, \dots, 1 }_{2n+\delta }, \qquad  \underbrace{2, \dots, 2}_{n(2n+1)}, \] where $\delta$ is given by \eqref{Eq:DeltaGLSL}. \end{theorem}

This case is similar to the case $\dim V = 2n$. We already know that for a form $Q \in \Lambda^2(V)$ with $\operatorname{rk} Q = 2k$ the dimension of stabilizer $\dim \operatorname{St}_Q$ is given by \eqref{Eq:DimStSkewGL} in the case of $\operatorname{GL}(V)$ and by \eqref{Eq:SLStabRankSkewEven} in the case of $\operatorname{SL}(V)$. The next statement is proved similarly to Propositions~\ref{Prop:SL_SkewEvenCong} and \ref{Prop:SL_SkewEvenCong}.

\begin{proposition}  \label{Prop:SL_SkewOddCong} Let $\rho$ be the representation of $\gl(V)$ or  $\operatorname{sl}(V)$  on the space of skew-symmetric forms $\Lambda^2(V)$, where $\dim V = 2n+1$, corresponding to the congruence action.

\begin{enumerate}

\item The singuar set is \begin{equation} \label{Eq:SingGLSLSkewOdd} \Sing = \left\{ \left. Q \right| \quad  \operatorname{rk} Q < 2n \right\}.\end{equation} 

\item Since $\codim \Sing > 1$ there are no Jordan blocks in the JK invariants. 

\item The numbers of horizontal and vertical indices are \[ \dimSt = n(2n+3) + \delta,  \qquad \codimO = 0 \] respectively, where $\delta$ is given by \eqref{Eq:DeltaGLSL}.

\end{enumerate}

\end{proposition}

In theory, one could calculate the JK invariants for any pair of forms $Q, B$. Yet it is much simpler to do it for generic pairs. 

\begin{proposition} \label{Prop:GenMatPair_GLSkewOdd} Any generic pair of skew-symmetric forms $Q$ and $B$ on an odd-dimensional space is congruent to the Kronecker block in the JK theorem: 
\[ A = \left(
\begin{array}{c|c}
  0 & \begin{matrix}
    1 &      0 &       &    \\
      & \ddots & \ddots &     \\
      &        &     1 & 0  \\
    \end{matrix} \\
  \hline
  \begin{matrix}
   -1&        &    \\
    0 & \ddots &    \\
      & \ddots & -1\\
     &        &  0 \\
  \end{matrix} & 0
 \end{array}
 \right) \quad  B=  \left(
\begin{array}{c|c}
  0 & \begin{matrix}
    0 &     1 &       &    \\
      & \ddots & \ddots &     \\
      &        &     0 & 1  \\
    \end{matrix} \\
  \hline
  \begin{matrix}
   0&        &    \\
    -1 & \ddots &    \\
      & \ddots & 0\\
     &        &  -1 \\
  \end{matrix} & 0
 \end{array}
 \right).
\]  \end{proposition}

\begin{proof}[Proof of Proposition~\ref{Prop:GenMatPair_GLSkewOdd} ] First, note that singular set $\Sing$, given by \eqref{Eq:SingGLSLSkewOdd}, has $\codim \Sing =3$. Indeed, by Proposition~\ref{Prop:SL_SkewOddCong} we have \[\dim  \mathcal O_{\mathrm{reg}} = \frac{\dim V(\dim V - 1)}{2}.\] Singular forms $Q\in \Sing$ satisfy $\rk Q \leq \dim V-3$, therefore by \eqref{Eq:DimStSkewGL} we have \[\dim \operatorname{St}_Q \geq \frac{\dim V(\dim V+1)}{2} + 3.\] Since $\dim \operatorname{GL}(V) = \left(\dim V\right)^2$, we have $\codim \Sing =3$. 

Next, for generic $Q, B \in \Lambda^2(V)$ we have $\dim \Ker (Q + \lambda B) = 1$ for all $\lambda$, since a generic line $Q + \lambda B$ doesn't intersect $\Sing$. Therefore the JK decomposition of a generic pair of skew-symmetric forms $Q, B$ consists of one Kronecker block. Proposition~\ref{Prop:GenMatPair_GLSkewOdd}  is proved. \end{proof}

\begin{corollary} \label{Cor:GenMatPair_SLSkewOdd} Consider the congruence action of $\operatorname{SL}(2n+1)$ on pairs of skew-symmetric forms $Q, B \in S^2(V^{2n+1})$. Any generic pair is congruent to some of the following pairs: 
\[ A = \left(
\begin{array}{c|c}
  0 & \begin{matrix}
    \mu &      0 &       &    \\
      & \ddots & \ddots &     \\
      &        &     \mu & 0  \\
    \end{matrix} \\
  \hline
  \begin{matrix}
   -\mu&        &    \\
    0 & \ddots &    \\
      & \ddots & -\mu\\
     &        &  0 \\
  \end{matrix} & 0
 \end{array}
 \right) \quad  B=  \left(
\begin{array}{c|c}
  0 & \begin{matrix}
    0 &     \mu &       &    \\
      & \ddots & \ddots &     \\
      &        &     0 & \mu  \\
    \end{matrix} \\
  \hline
  \begin{matrix}
   0&        &    \\
    -\mu & \ddots &    \\
      & \ddots & 0\\
     &        &  -\mu \\
  \end{matrix} & 0
 \end{array}
 \right).
\]   \end{corollary}

Now we can find the number of trivial Kronecker blocks.

\begin{proposition} 
\label{Prop:SkewOddForm_GL_TrivKronBlocks}
Consider a generic pencil $R_{Q }+\lambda R_B$ for the representation of $\operatorname{gl}(V)$ or $\operatorname{sl}(V)$ on the space of skew-symmetric forms $\Lambda^2(V)$, with $\dim V = 2n+1$, corresponding to the congruence action. Then the number of horizontal and vertical indices equal to $1$ is
\[ \dim \left(\operatorname{St}_Q \cap \operatorname{St}_B\right) = 2n+\delta, \qquad  \codim \left(\operatorname{Im} R_Q + \operatorname{Im} R_B \right)  = 0\] respectively, where $\delta$ is given by \eqref{Eq:DeltaGLSL}.\end{proposition}

\begin{proof}[Proof of Proposition~\ref{Prop:SkewOddForm_GL_TrivKronBlocks}] $ \codim \left(\operatorname{Im} R_Q + \operatorname{Im} R_B \right)  = 0$ since $\codimO = 0$. It remains to calculate $ \dim \left(\operatorname{St}_Q \cap \operatorname{St}_B\right)$. The elements $C \in \operatorname{GL}(V)$ that preserve both forms $Q$ and $B$: \[ C^T Q C = Q, \qquad C^T B C = B,\] are automorphisms of the pencil $Q + \lambda B$ and they were described in \cite{Pumei14}. The Lie algebra $\operatorname{aut}(V, Q, B)$ of such automorphisms are given by the equations \[ C^T Q + QC = 0, \qquad C^T B + BC = 0.\] For one Kronecker block this Lie algebra $\operatorname{aut}(V, Q, B)$ consists of the matrices 
\begin{equation} \label{E:BiSymp_Same_Kronecker_Blocks} \left(
\begin{array}{cccc|ccccc}
\alpha^{-1} & & & & & & & &\\
  & \ddots & & & & & & &\\
   & & \ddots & & & & & &\\
    & & & \alpha^{-1} & & & & &\\
    \hline
     x_{1} & x_2& \cdots & x_{n} & \alpha & & & &\\
      x_2 & \udots & \udots & x_{n+1}  & & & & &\\
       \vdots & \udots & \udots & \vdots & & & \ddots & &\\
        x_k & \udots & \udots & \vdots & & & & \ddots &\\
         x_{n+1} & \cdots & \cdots & x_{2n} & & & & & \alpha \\ \end{array}\right).\end{equation} The elements of $\operatorname{St}_Q \cap \operatorname{St}_B$ are the transposed matrices $C^T$, where $C$ is given by \eqref{E:BiSymp_Same_Kronecker_Blocks}. We see that the dimension of matrices in \eqref{E:BiSymp_Same_Kronecker_Blocks} is $2n+1$ for $\gl(2n+1)$. For $\operatorname{sl}(2n+1)$ the trace must be zero, thus the dimension is $2n$. Hence, $\dim \left(\operatorname{St}_Q \cap \operatorname{St}_B\right) = 2n+\delta$.  Proposition~\ref{Prop:SkewOddForm_GL_TrivKronBlocks} is proved. \end{proof}

Now we prove the statement about the JK invariants for the differentials of the conguence actions of $\operatorname{GL}(2n+1)$ and $\operatorname{SL}(2n+1)$ on skew-symmetric forms $\Lambda^2(V^{2n+1})$.

\begin{proof}[Proof of Theorem~\ref{Prop:SL_SkewOddCong}] Theorem~\ref{Prop:SL_SkewOddCong}  follows from Propositions~\ref{Prop:SL_SkewOddCong} and \ref{Prop:SkewOddForm_GL_TrivKronBlocks} and considerations of dimensions.\end{proof}

\section{Problems}

All irreducible representations of $\operatorname{sl}(2)$ are well-known. A simple problem for an under-graduate student is 

\begin{problem}
Describe the JK invariants for all the irreducible representations of $\operatorname{sl}(2)$.
\end{problem}

One may wonder: what JK invariants for Lie algebra representations are possible? For low-dimensional algebras it should be easy to answer this question completely. 

\begin{problem} \label{Prob:LowDim} Describe all possible JK invariants for representations of low-dimensional Lie algebras $\rho: \mathfrak{g} \to \operatorname{gl}(V)$, e.g. for $\dim \mathfrak{g} = 2$ or $3$. 
\end{problem}

Corollary~\ref{Cor:JKInv_BigRepr} can be useful, when solving Problem~\ref{Prob:LowDim}. One can ask a more general question.

\begin{problem} \label{Prob:WhatJKPossible}
What JK invariants of linear pencils can be realized by Lie algebra representations? 
\end{problem}

For Lie algebras the JK invariants of a sum $\mathfrak{g}_1 \oplus \mathfrak{g}_2$ is the union of JK invariants of Lie algebras $\mathfrak{g}_1$ and $\mathfrak{g}_2$. For Lie algebra representations we have the following similar statement.

\begin{proposition} \label{Prop:UnionRepr}
Let $\rho_i : \mathfrak{g}_i \to \operatorname{gl}(V_i)$, $i=1,2$, be Lie algebra representations. Then the JK invariants of the representation \begin{gather*} \rho: \mathfrak{g}_1 \oplus \mathfrak{g}_2 \to \operatorname{gl}(V_1) \oplus \operatorname{gl}(V_2), \\ \rho(x_1, x_2) = \left( \rho_1 (x_1), \rho_2(x_2)\right),\end{gather*}
is the union of the JK invariants of $\rho_1$ and $\rho_2$.
\end{proposition}

Proposition~\ref{Prop:UnionRepr} may help to construct new JK invariants, when solving Problem~\ref{Prob:WhatJKPossible}. In general, standard operations with representations do not interact well with the JK invariants.

\begin{problem} Let $\mathfrak{h} \subset \mathfrak{g}$ be a subalgebra with $\codim \mathfrak{h} = 1$. Consider a restriction $\rho\bigr|_{\mathfrak{h}}$ of a Lie algebra representation $\rho: \mathfrak{g} \to \operatorname{gl}(V)$ on $\mathfrak{h}$. What combinations of JK invariants for $\rho$ and $\rho\bigr|_{\mathfrak{h}}$ are possible? What is the ``generic case''? \end{problem}

A slightly different problem would be to calculate the JK invariants for all pencils $R_{X} + \lambda R_A$, not only generic one. Here the comparison of $\operatorname{gl}(n)$ (where the answer is known, see Theorem~\ref{T:JK_RightMatrixAction}), and  $\operatorname{sl}(n)$ can be interesting.

\begin{problem}
What are the JK invariants of any pencil $R_{X} + \lambda R_A$ for the sum of standard representations of $\operatorname{sl}(n)$?
\end{problem}

Also note, that we know the canonical form for any pair $X, A$ for the congruence actions of  $\operatorname{GL}(n)$ and $\operatorname{SL}(n)$ on symmetric forms and skew-symmetric forms (see Section~\ref{S:CongAct}). But the calculations are probably cumbersome and not worth it.

\begin{problem}

There are a lot of representations for which one can try to calculate the Jordan--Kronecker invariants. Just to name a few:

\begin{enumerate}

\item sum of standard representations of the Lie algebra of strictly upper triangular matrices $\operatorname{n}(n)$.

\item sum of standard representations of Borel subalgebras of semisimple Lie algebras.

\item differential of the congruence  action for the Lie algeras $\operatorname{so}(n)$ and $\operatorname{sp}(n)$ on symmetric forms and skew-symmetric forms.

\item adjoint and coadjoint representations of classical matrix Lie algebras.

\end{enumerate}

\end{problem}

Invariants for adjoint representation of a simple Lie algebra $\mathfrak{g}$ are well-known and they are related with exponents of $\mathfrak{g}$. For example, invariants of the algebra of invariants for the  adjoint representation  of $\operatorname{sl}(n)$ is generated by $\operatorname{tr}\left(M^k\right)$, $k=2, \dots, n$. A generic matrix is semisimple, hence it is similar to a diagonal matrix. Thus, there are $n-1$ horizontal and $n-1$ vertical indices:  \[ \dimSt = n-1, \qquad  \codimO = n-1.\] By Theorem~\ref{thm3} the vertical indices are equal to the exponents $v_i = 2, \dots, n$. 

So far, we do not have many statements about horizontal indices. In terms of matrices, horizontal blocks are transposed vertical, and all the statements should be similar. For representations, vertical indices are related with invariants of representations and horizontal indices should be connected with ``a span of stabilizers''. 

\begin{problem}
Are there results similar to Theorems~\ref{thm1},\ref{thm2} and \ref{thm3} about horizontal indices?
\end{problem}


\begin{thebibliography}{99}


\bibitem{BolsIzosKozl19} A.\,V.~Bolsinov, A.\,M~Izosimov, I.\,K.~Kozlov, ``Jordan-Kronecker invariants of Lie algebra representations and degrees of invariant polynomials'', \textit{Transformation Groups} (2021). https://doi.org/10.1007/s00031-021-09661-0 

\bibitem{BolsZhang}
A.\,V.~Bolsinov, P.~Zhang, ``Jordan-Kronecker invariants of finite-dimensional Lie algebras'',
\textit{Transformation Groups}, \textbf{21}:1 (2016),  51--86.

\bibitem{Gantmaher88}
F.~Gantmacher, {\it Theory of matrices}, Vols. 1, 2, AMS Chelsea publishing, New York, 1959.


\bibitem{Gar1} A.\,A.~Garazha, ``A canonical basis of a pair of compatible Poisson brackets on a matrix algebra'', \textit{Sb. Math.}, \textbf{211}:6 (2020), 838-849

\bibitem{Gar2} A.\,A.~Garazha, ``On a canonical basis of a pair of compatible Poisson brackets on a symplectic Lie algebra'', \textit{Uspekhi Mat. Nauk}, \textbf{77}:2(464) (2022), 199-200; \textit{Russian Math. Surveys}, \textbf{77}:2 (2022), 375-377

\bibitem{Joseph10} A.~Joseph, D.~Shafrir, ``Polynomiality of invariants, unimodularity and adapted pairs'',
\textit{Transform. Groups}, \textbf{15} (2010), no. 4, 851-882

\bibitem{Pokrzywa86}  A.~Pokrzywa, ``On perturbations and the equivalence orbit of a matrix pencil'', \textit{Linear
Algebra and its Applications}, \textbf{82} (1986), 99-121.


\bibitem{Thompson}
R.\,C.~Thompson, ``Pencils of complex and real symmetric and skew matrices'',
\textit{Linear Algebra and its Applications}, \textbf{147} (1991), 323--371.

\bibitem{Voron} A.\,S.~Vorontsov, ``Kronecker indices of Lie algebras and invariants degrees estimate'', \textit{Moscow
University Math. Bulletin}, \textbf{66}: 1 (2011), 25-29

\bibitem{Vor1} K.\,S.~Vorushilov, ``Jordan-Kronecker invariants for semidirect sums defined by standard
representation of orthogonal or symplectic Lie algebras'', \textit{Lobachevskii Journal of
Mathematics}, \textbf{36}:6 (2017), 1121-1130

\bibitem{Vor2} K.\,S.~Vorushilov, ``Jordan-Kronecker invariants of semidirect sums of the form $\operatorname{sl}(n)+(R^n)^k$ and $\operatorname{gl}(n)+(R^n)^k$'', \textit{Fundam. Prikl. Mat.}, \textbf{22}:6 (2019),  3-18

\bibitem{Vor3} K.\,S.~Vorushilov, ``Complete sets of polynomials in bi-involution on nilpotent seven-dimensional Lie algebras'', \textit{Sb. Math.}, \textbf{212}:9 (2021), 1193-1207


\bibitem{Vor4} K.\,S.~Vorushilov,  ``Jordan-Kronecker invariants of Borel subalgebras of semisimple Lie algebras'', \textit{Chebyshevskii Sb.}, \textbf{22}:3 (2021), 32-56

\bibitem{Pumei14} Pumei Zhang, ``Algebraic Properties of Compatible Poisson Brackets'', \textit{Regul. Chaotic Dyn.}, \textbf{19}:3 (2014), 267-288.


\end{thebibliography}
\end{document}